\documentclass[%
 aip,
 jcp,%
 amsmath,amssymb,
reprint,%
author-numerical,%
floatfix
]{revtex4-1}

\usepackage{amsmath, amssymb, amsthm,amsfonts, amsbsy, amsxtra, latexsym}
\usepackage[hidelinks]{hyperref}
\usepackage{array}
\usepackage{enumerate, graphicx,color}
\usepackage{fullpage}
\usepackage{epstopdf}

\usepackage{tikz}
\usetikzlibrary{arrows, backgrounds, snakes}
\usepackage{verbatim}
\usepackage{subcaption}

\newcommand*{\citen}{}
\DeclareRobustCommand*{\citen}[1]{%
  \begingroup
    \romannumeral-`\x 
    \setcitestyle{numbers}%
    \cite{#1}%
  \endgroup
}

\newcommand*{\mathcolor}{} 
  \def\mathcolor#1#{\mathcoloraux{#1}}
  \newcommand*{\mathcoloraux}[3]{%
    \protect\leavevmode
    \begingroup
      \color#1{#2}#3%
      \endgroup
    }

 \newcommand{\E}{\mathbb{E}} 
\newcommand{\Prob}{\mathbb{P}} 
\def\Var{\mathbb{V}\mathrm{ar}} 
\def\Exp{\mathcal{E}{{xp}}} 
\newcommand{\indic}{\mathbb{I}} 
\newcommand{\one}{\boldsymbol{1}} 
\newcommand{\defd}{\overset{\Delta}{=} } 
\newcommand{\rr}{\mathbb{R}} 
\newcommand{\diag}{ {\operatorname{diag} } } 
\newcommand{\NC}{ {N_{C}}} 
\newcommand{\NS}{ {N_{S}}} 
\newcommand{\M}{\mathcal{M}} 
\newcommand{\ep}{\varepsilon}

\newcommand{\btheta}{{\boldsymbol{\theta}}} 
\newcommand{\balpha}{ { {\boldsymbol{\alpha} } } } 
\newcommand{\bbeta}{ { {\boldsymbol{\beta} } } } 

\newtheorem{theorem}{Theorem}[section]
\newtheorem{assumption}[theorem]{Assumption}
\newtheorem{proposition}[theorem]{Proposition}
\newtheorem{lemma}[theorem]{Lemma}

\newtheorem{corollary}[theorem]{Corollary}

\newtheorem{algorithm}{Algorithm}

\begin{document}
\title{Stochastic Averaging and Sensitivity Analysis for Two Scale Reaction Networks}
\thanks{Submitted to the Journal of Chemical Physics.}

\author{Araz Hashemi}
\email[Corresponding author, ]{araz@udel.edu}
\affiliation{Department of Mathematical Sciences, University of Delaware}

\author{Marcel Nunez}
\email{mpnunez@udel.edu}
\affiliation{Department of Chemical and Biomolecular Engineering, University of Delaware}
 
\author{Petr Plech\'a\v{c}}
\email{plechac@udel.edu}
\affiliation{Department of Mathematical Sciences, University of Delaware}

\author{Dionisios G. Vlachos}
\email{vlachos@udel.edu}
\affiliation{Department of Chemical and Biomolecular Engineering, University of Delaware}

\date{\today}
\begin{abstract}
In the presence of multiscale dynamics in a reaction network, 
direct simulation methods become inefficient as they 
can only advance the system on the smallest scale. 
This work presents stochastic averaging techniques
to accelerate computations 
for obtaining estimates of expected 
values and sensitivities with respect to the steady state distribution. 
A two-time-scale formulation is used to establish bounds 
on the bias induced by the averaging method.
Further, this formulation provides a framework 
to create an accelerated `averaged' version of most
single-scale sensitivity estimation methods.
In particular, we propose a new lower-variance ergodic likelihood ratio method
for steady state 
estimation and show how one can adapt it to accelerate simulations of multiscale systems.
{Lastly, we develop an adaptive ``batch-means'' stopping rule for determining when to terminate
the micro-equilibration process.}
 
\end{abstract}


\maketitle

\section{Introduction}

Stochastic simulations have been an essential tool 
in analyzing reaction networks encountered in
biology, catalysis, and materials growth. 
However, it is commonplace for reaction networks to exhibit a large disparity in time scales. 
These multi-scale stochastic reaction networks can impose
an enormous computational burden in order to simulate them exactly.
Exact techniques require computation of every reaction 
at the fastest time-scale, resulting in an exponentially
growing load to observe dynamics on the slowest time-scale.
Many works have attempted to develop approximate
algorithms which allow faster computation with minimal loss of accuracy
\cite{chatterjee_overview_2007, 
gillespie_approximate_2001, cao_avoiding_2005,
cao_efficient_2006, rathinam_stiffness_2003,
chatterjee_binomial_2005,tian_binomial_2004,
salis_accurate_2005,
samant_overcoming_2005,cao_slow-scale_2005, e_nested_2007,
huang_strong_2014,
kang_separation_2013, gupta_sensitivity_2014}.

One approach,
which we refer to as Stochastic Averaging,
takes its inspiration from classical singular perturbation theory of
ordinary differential equations
\cite{samant_overcoming_2005,cao_slow-scale_2005, e_nested_2007, huang_strong_2014, salis_accurate_2005}.
The idea is that the fast dynamics come to quasi-equilibrium before
the slow dynamics take effect, 
hence one may model the slow time scale dynamics
with their averages against the steady-state 
distribution of the fast dynamics.
By estimating the steady-state expectations of the slow propensities, 
one can then 
jump the system ahead to the next slow reaction
and advance the time clock on the
slow scale (skipping over needless computations of fast reactions).

In addition, one often desires the sensitivities of the system 
$S_f(\theta_i) = \frac{\partial}{\partial \theta_i} \E_{\btheta} \{ f(X(t)) \}$
with respect to
the reaction parameters $\theta_i$.
The sensitivities give important insight into the system,
indicating directions for gradient-descent type optimization as well as 
determining bounds for quantifying the uncertainty 
\cite{dupuis_path-space_2015}.
Current techniques for estimating the sensitivities 
have {\em large variance},
requiring many more samples than those for estimating
$\E_{\btheta} \{f(X)\}$ alone 
\cite{wolf_hybrid_2015,sheppard_pathwise_2012,
wang_efficiency_2014, gupta_efficient_2014}.
Thus computing sensitivities of multi-scale systems 
using single-scale
techniques is often a computationally intractable problem.

{
In this work,
we use results from Two-Time-Scale (TTS) Markov chains
\cite{yin_continuous-time_2013} 
to show the error 
of stochastic averaging algorithms
is inverse to the scale disparity in the system.
As opposed to the previous approaches of transforming 
the system variables into auxiliary fast and slow 
variables\cite{e_nested_2007, huang_strong_2014}, 
we partition the underlying (discrete) state space 
and derive a singular perturbation expansion 
of the corresponding probability measure. 
The first order term can then be identified 
from computables of the averaged process, 
leading to a rigorous theoretical framework for applying 
singular perturbation averaging for stochastic systems.

Furthermore, this new formulation allows one to identify a 
macroscopic ``averaged'' reaction network 
on a reduced state space 
whose time-steps are on the macro (slow) time-scale. 
Thus, it provides a framework for applying single-scale 
sensitivity analysis techniques to the multi-scale system. 
Previous works have exploited similar model
reduction techniques to estimate
sensitivities via finite differences\cite{gupta_sensitivity_2014}
or a ``truncated'' version of a likelihood ratio
estimator\cite{nunez_steady_2015}.
This work develops an accelerated
``Two-Time-Scale'' version of the 
Likelihood Ratio (Girsanov Transform) 
method \cite{wang_efficiency_2014, plyasunov_efficient_2007, 
warren_steady-state_2012, glynn_likelihood_1990}
for estimating system sensitivities of the multiscale system.
The TTS-LR method computes sensitivity reweighting coefficients
 of the macro (reduced-state) process using a representation 
in terms of the steady-state sensitivities of the micro (fast) process.
These micro-level sensitivities can in turn be computed online during
the micro-equilibriation process.
To this end, we propose a new 
lower-variance 
``Ergodic Likelihood Ratio''
estimator for approximating {\em steady-state sensitivities} of single and multi-scale systems.
}

{
The outline of the remainder of the paper is as follows:
Section \ref{sec:formulation} gives the theoretical basis of the paper.
The Two-Time-Scale formulation is presented and error bounds are established.
Section \ref{sec:TTS_sens} then uses the TTS framework for the purpose of sensitivity analysis.
A new ergodic likelihood ratio estimator is developed for single-scale steady-state
sensitivity analysis, and is then adapted to the multiscale system.
Section \ref{sec:BMstop} develops a batch-means stopping rule 
for determining when the micro-scale
system has come to equilibrium.
Numerical results are presented in Section \ref{sec:sim_results} supporting
the effectiveness of the methods presented.
Concluding remarks are given in Section \ref{sec:conclusion},
and proofs of theorems are relegated to the Appendix.
}

\section{Formulation}
\label{sec:formulation}

\subsection{Markov Chain Model of Reaction Networks}
We briefly review the Markov chain model of reaction networks.
While our motivation stems from chemical reaction networks, we note that 
much of the formulation carries over to general Markov chains on integer lattices.

Suppose we have $d$ species
described by 
$X(t) = [X_1(t), X_2(t), \dots, X_d(t)] \in \M \subset \mathbb{Z}^d$
and $M$ reactions $r_1, r_2, \dots, r_M$. 
In stochastic reaction networks, one views 
$X(t)$ as a continuous-time Markov chain (CTMC) in the state space $\M$.
When reaction $r$ fires at time $t$, the state is updated by the 
{\em stoichiometric vector} $\zeta_r$ so that 
$X(t) = X(t-) + \zeta_r$.
Given the set of reaction parameters $\btheta=[\theta_1, \theta_2, \dots]$,
one characterizes the probabilistic evolution of $X(t)$ by the
{\em propensity functions} (intensity functions) $\lambda_r(x;\btheta)$.
The propensity functions are such that, given $X(t)=x$, the probability of one or
more firing of reaction $r$ during time $(t, t+h]$ is 
$\lambda_r(x;\btheta)h + o(h)$ as $h \to 0$; i.e. $\lambda_r(x;\btheta)$ is the
instantaneous rate/probability of reaction $r$ firing.

A common model for the propensities functions is that of {\em mass-action kinetics}.
Under this assumption, the propensity functions are of the form
\begin{align} \label{eq:mass-action-prop}
  \lambda_r(x;\btheta) = \theta_r \cdot b_r(x) = 
  \theta_r \cdot \prod_{i=1}^d \frac{x_i!}{(x_i - \nu_{r,i} )!} 
  \indic_{\{x-\nu_{r,i} \ge 0\} }
\end{align}
where $\nu_{r,i}$ is the number of molecules of species $i$ required for reaction $r$
to fire. Mass action kinetics assumes the system is well-mixed,
so molecular interactions are proportional to their counts. 

From the propensity functions $\lambda_r(x;\btheta)$, one can construct the 
infinitesimal generator $Q=Q(\btheta)$ of the Markov chain. Viewed as an operator
on functions $f$ of the state-space $\M$, we have
\begin{align*}
  \left( Qf \right)(x) = \sum_{r=1}^M \lambda_r(x;\btheta) 
  \left( f(x+\zeta_r) - f(x)  \right)
\end{align*}
For finite state-space $\M$ (bounded molecule counts), one may also view
the generator $Q(\btheta)$ as a matrix. 
While the state-space is typically intractably large, the generator is sparse with only $M+1$
non-zero entries in each row,
\begin{multline*}
	  Q(\boldsymbol{\theta}) = \\ \bordermatrix{ ~
        & \dots & x+\zeta_2 & \dots & x & \dots & x+\zeta_1 & \dots  \cr 
        \vdots & ~ & ~ & ~ & ~ & ~ & ~ & ~ \cr
	x & \dots & \lambda_2(x;\boldsymbol{\theta}) & \dots & -\lambda_0(x;\boldsymbol{\theta})
	& \dots & \lambda_1(x;\boldsymbol{\theta)} & \dots \cr
      \vdots & ~ & ~ & ~ & ~ & ~ & ~ & ~ }
      \end{multline*}
where 
$\lambda_0(x;\boldsymbol{\theta}) = \sum_{r=1}^M \lambda_r(x;\boldsymbol{\theta})$.   

Writing $R_r(t)$ to be the counting process 
representing how many reactions of type $r$
have fired by the time $t$, we have that
$X(t) = X(0) + \sum_{r=1}^M R_r(t) \zeta_r$.
Using the {\em random time change} representation 
\cite{anderson_continuous_2011, ethier_markov_1986},
we write $X(t)$ as 
\begin{align}
  \begin{aligned}
  \label{eq:random-time-change}
    X(t) = X(0) + \sum_{r=1}^M Y_r\left( 
    \int_0^t \lambda_r(X(s);\btheta)ds \right) \zeta_r
  \end{aligned}
\end{align}
where $Y_r(\cdot)$ are independent unit-rate Poisson processes.
This representation is tremendously useful in conducting analysis
of the trajectories.
In particular, it leads to formulations of the Next-Reaction Method
\cite{gibson_efficient_2000, anderson_modified_2007} 
and interpreting simulated trajectories in 
the path-space to allow for coupling paths 
\cite{rathinam_efficient_2010, anderson_efficient_2012, gupta_efficient_2014}
as well as path-wise differentiation 
\cite{sheppard_pathwise_2012, wolf_hybrid_2015}.

When simulating exact trajectories (using any exact method; Direct SSA, Next-Reaction,
etc), the propensity functions $\lambda_r(x; \btheta)$ probabilistically determine
both the time between reactions $\Delta t$ 
as well as the next reaction $r^*$ to fire.
The likelihood that reaction $r_k$ is the next to fire is proportional to its
propensity $\lambda_k(x;\btheta)$; i.e. $\Prob_{\btheta, x} \left\{ r^* = r_k  \right\} 
\propto \lambda_k(x;\btheta) $. 
The time between reactions has an exponential
distribution with the rate $ \lambda_0(x;\btheta) 
= \sum_{r=1}^M \lambda_r (x;\btheta) $ ; i.e. $\Delta t \sim \Exp \left( \lambda_0(
x; \btheta)  \right) $ with the mean 
$\E_{x,\btheta} \{ \Delta t \} = 1/ \lambda_0(x,\btheta)$.

Multi-scale dynamics occur when the propensity functions have large magnitude disparities.
If $\lambda_k(x; \btheta) \gg \lambda_j(x;\btheta)$ for all $j \neq k$, then
$\Prob \{ r^*=r_k \} \approx 1$ and 
$\Delta t \sim \Exp \left( \lambda_0(x;\btheta) \right) \approx 
\Exp \left( \lambda_k(x;\btheta) \right) $. 
Thus, with a high probability the next reaction in an exact trajectory will be $r_k$ and the
time clock will advance on the order of $1/\lambda_k(x;\btheta)$. 
Such multi-scale networks then require an enormous 
number of computations to sample ``slow'' reactions 
and reach the required time horizon for the entire system to relax.

\subsection{Two-Time-Scale Reaction Networks}
\label{sec:TTSRN}
{ 
We now consider reaction networks with two scales of dynamics.
For further motivation and discussion of 
reaction networks with multiple time-scales, 
we refer readers to
Refs. \citen{kang_separation_2013, e_nested_2007, huang_strong_2014, gupta_sensitivity_2014}
and references therein.
We instead focus on our formulation for the separation of time-scales
and the averaged process via the partitioning of the state space into
``fast-classes''.
Though analogous to the techniques of transforming the species variables
into auxiliary fast/slow variables \cite{e_nested_2007, huang_strong_2014}
or projecting to remainder spaces \cite{gupta_sensitivity_2014},
the direct partitioning of the state space 
will allow us to construct
a singular perturbation expansion of the probability measure
and establish the rate of convergence of such averaging methods.
In addition, it provides a framework for applying Likelihood Ratio type
sensitivity estimates to the averaged process as we shall see in the sequel.

Here, we assume that the disparity in the propensity functions 
results from magnitude disparities in the reaction parameters $\theta_r$.
In order to illustrate the stiffness, we consider the reaction network 
  \begin{align*}
    \begin{aligned}
      *
     \mathrel{\mathop{\rightleftharpoons}^{ {\alpha}_1 / \ep}_{\mathrm{ {\alpha}_2 / \ep}}} 
    A & \qquad &
    A
     \mathrel{\mathop{\rightleftharpoons}^{ {\beta}_1}_{\mathrm{{\beta}_2}}} 
    B & \qquad & B 
     \mathrel{\overset{ {\beta}_3}{\rightharpoonup}}
    * 
    \end{aligned}
  \end{align*}
where $\ep \ll 1$ is a measure of the scale disparity (stiffness) between the fast reaction parameters 
$\balpha=[\alpha_1, \alpha_2] $ and the slow reaction
parameters $\bbeta=[\beta_1, \beta_2, \beta_3 ]$.
As the stiffness parameter $\ep \to 0$ the fast reactions $\balpha$ dominate the system,
resulting in the multi-scale computational problem described above.

In general, 
suppose a reaction network has species $[X_1, X_2, \dots,, X_d]$
and reactions $r_1, \dots, r_M$. 
We shall assume that the propensity functions $\lambda_r(x;\btheta)$ 
are of the form (\ref{eq:mass-action-prop}) (mass-action kinetics, 
though other forms may also be treated), 
and that
each reaction is indexed by its own reaction parameter $\theta_r$.
As in the illustrating example,
we assume that there is a scale disparity in the reaction parameters
between a set of 
``fast reactions'' and a set of 
``slow reactions''. 
Thus we can write
$\btheta = [\theta_1, \dots, \theta_M]
= [\balpha/\ep, \bbeta]$, where
$\bbeta=[ \beta_1, \beta_2, \dots, \beta_{M_s}]$
are the slow reaction parameters,
$\ep \ll 1$ is the stiffness parameter,
and $\balpha=[\alpha_1, \dots, \alpha_{M_f}] $
are the underlying (rescaled) reaction parameters 
for the fast reactions.

}

To ease referencing, we will often index reactions and propensity functions
directly by their reaction parameter. E.g.,  
$r_{\beta_i}$ is the reaction with reaction parameter $\beta_i$
and propensity function $\lambda_{\beta_i} (x; \btheta)= \lambda_{\beta_i}(x;\bbeta) 
= \beta_i 
b_{\beta_i}(x) $ (where $b_{\beta_i}$ is given by 
\eqref{eq:mass-action-prop}).
For the fast reactions $\alpha_i$, we use 
$\lambda^\ep_{\alpha_i}(x;\btheta) 
=\lambda^\ep_{\alpha_i}(x;\balpha)
= ({\alpha_i}/ {\ep}) 
b_{\alpha_i}(x)$ to denote the exact propensity function
and $\lambda_{\alpha_i}(x;\btheta) = \alpha_i \  b_{\alpha_i}(x)$ 
to denote the rescaled version. 

Let $X^\ep(t)$ denote the Markov chain determined by 
the exact propensity functions
$\lambda^\ep_\alpha(x;\btheta)$ and $\lambda_\beta(x;\btheta)$.
We can write the generator $Q^\ep=Q^\ep(\balpha,\bbeta)$ of the exact process
as before, and observe that
\begin{align}
\begin{aligned}
\label{eq:Q-ep_decomp}
(Q^\ep f) (x)
&= \sum_{r=1}^M \lambda_r(x;\btheta) 
\left( f(x+\zeta_r) -f(x)  \right) \\
&= \frac{1}{\ep} \sum_{i=1}^{M_f} \lambda_{\alpha_i}(x;\balpha) 
\left( f(x+\zeta_{\alpha_i}) -f(x)  \right) \\
& \qquad + \sum_{j=1}^{M_s} \lambda_{\beta_j}(x;\bbeta)
\left( f(x+\zeta_{\beta_j}) -f(x)  \right) \\
&= \left( \left[ \frac{1}{\ep} \widetilde{Q}(\balpha) + 
\widehat{Q}(\bbeta) \right] f \right) (x)
\end{aligned}
\end{align}
where $\widehat{Q}(\bbeta)$ is the generator of the chain under only 
the slow dynamics (determined by slow reactions $r_\beta$), 
and $\widetilde{Q}(\balpha)$ is the generator of the chain under
only the fast dynamics (with the rescaled propensity functions
$\lambda_\alpha(x;\balpha)$).
Thus we have a decomposition of the generator into the fast
and slow dynamics, 
$Q^\ep(\balpha, \bbeta) = (1/\ep) \widetilde{Q}(\balpha) + 
\widehat{Q}(\bbeta) $
One can also view the generator 
$Q^\ep(\balpha, \bbeta) = (1/\ep) \widetilde{Q}(\balpha) + \widehat{Q}(\bbeta)$
as a matrix. In this case, we can write the element corresponding 
to the rate of transition
from state $x$ to state $y$ as
\begin{align*} 
  \left[ \widetilde{Q}(\boldsymbol{\alpha})\right]_{x,y}
  =\left\{
    \begin{array}{*1{>{\displaystyle}c }l}
    -\sum_{\alpha \in \boldsymbol{\alpha}} \lambda_\alpha(x)
    & \quad x=y \\
    \sum_{\alpha : r_\alpha(x)=y } \lambda_\alpha(x)
    & \quad x \neq y
  \end{array}
  \right.
\end{align*}
and similarly for 
$ \left[ \widehat{Q}(\boldsymbol{\beta})\right]_{x,y}$.

As $\ep \to 0$ only the fast reactions $r_\alpha$ fire, and so we
define an equivalence relation on states $s \in \M$,
by $s_i \leftrightarrow s_j$ if they are mutually accessible
through only fast reactions. 
This defines a partition of the state space $\M$
into ``fast-classes'' $\M_k$ which are by construction the
invariant (irreducible) classes of $\M$ under $\widetilde{Q}(\balpha)$;
e.g.
\begin{align*}
  \M= \bigcup_{k=1}^\NC \M_k =
  \left\{ x^{(1)}_1, x^{(1)}_2, \dots, x^{(1)}_{m_1},
  x^{(2)}_1, \dots, x^{(2)}_{m_2},\dots
\right\}
\end{align*}
where $\NC$ are the number of invariant ``fast-classes'', 
and $m_k = \big| \M_k \big|$ is the number of states inside
fast-class $\M_k$.
For ease of presentation, in the present discussion we shall assume the state space is finite.
\begin{assumption}[Finite State Space] \label{assum:finite_states}
The state space $\M$ is finite, such that $|\M| =m$.
Thus the number of fast classes $\NC<\infty$ and 
the number of states in each fast-class $m_k<\infty$,
so that $m=m_1 + m_2 + \dots + m_\NC$.
\end{assumption}
Assumption~\ref{assum:finite_states} is made only to simplify the discussion.
One may also treat the infinite state case with some mild 
additional conditions on $\widetilde{Q}(\balpha)$ and 
$\widehat{Q}(\bbeta)$ to ensure non-explosiveness and ergodicity of the 
underlying (rescaled) chain\cite{yin_continuous-time_2013}.
In addition, we shall impose the following assumption.
\begin{assumption}[Recurrent States] \label{assum:recurr_states}
  Each state of $\M$ is recurrent, so that there are no absorbing/transient
  states.
\end{assumption}
Assumption~\ref{assum:recurr_states} is satisfied if, for example, all 
reactions are reversible 
(or often times if only the fast reactions are reversible).
One may also treat the case with transient/absorbing classes 
with some additional 
stability assumptions to ensure the fast dynamics decay to steady-state; 
see Section 4.4 of Ref.
\citen{yin_continuous-time_2013}
for more details.

Under Assumption~\ref{assum:recurr_states}, we can reorder the state space so that
$\widetilde{Q}(\balpha)=\diag [ \widetilde{Q}^{(1)}(\balpha),
\widetilde{Q}^{(2)}(\balpha), \dots, 
\widetilde{Q}^{(\NC)}(\balpha)  ]$ 
is block-diagonal. Here, one can view the generators
$\widetilde{Q}^{(k)} (\balpha)$ as the restriction of $\widetilde{Q}(\balpha)$
to the (irreducible) fast-class $\M_k$ (fast-only dynamics when $X(0) \in \M_k$).
In light of the finite state-space and positive recurrence, each
$\widetilde{Q}^{(k)}(\balpha)$ is ergodic and has a stationary (steady-state)
probability measure $\widetilde{\pi}^{(k)} = \widetilde{\pi}^{(k)}(\balpha)$ such that 
$\widetilde{\pi}^{(k)} \widetilde{Q}^{(k)} = \boldsymbol{0}$ 
(with $\widetilde{\pi}^{(k)}, \boldsymbol{0}$ 
interpreted as row vectors).

Using the above formulation, we can restate the averaging 
principle \cite{e_nested_2007, cao_slow-scale_2005, samant_overcoming_2005, huang_strong_2014,
kang_separation_2013, gupta_sensitivity_2014}
as follows.
For small $\ep$ and $X(0) \in \M_k$, $X^\ep(\cdot)$ will relax to 
its steady-state distribution $\pi^{(k)}$ on the micro 
time-scale $\ep t$ before any slow reaction fires
(on the macro time scale $t$).
Thus, one can use the stationary average of the slow propensity functions
\begin{multline}
    \lambda_{\beta_j} \left( X 
    \left( t\right); \bbeta, X(0) \in \M_k \right) \\
    \sim \overline{\lambda}_{\beta_j} \left( \M_k; \btheta \right)
    \defd \E_{\widetilde{\pi}^{(k)}(\balpha)} \left\{ \lambda_{\beta_j} 
  \left( X; \bbeta \right) \right\}
\end{multline}
to determine the distribution of time until the next slow reaction as well
as the probabilities for the next slow reaction being $r_{\beta_j}$.
These can then be used to simulate a trajectory of the slow (macro-scale) 
process. 
We shall further develop this idea more precisely in the remainder.

Write $m=\big| \M \big|$, and $m_k= \big| \M_k \big|$ as before, so that 
$Q^\ep(\btheta), \widetilde{Q}(\balpha), \widehat{Q}(\bbeta) \in 
\mathbb{R}^{m \times m}$, 
$\widetilde{Q}^{(k)}(\balpha) \in \mathbb{R}^{m_k \times m_k}$, 
and $\pi^{(k)}(\balpha) \in \mathbb{R}^{1 \times m_k}$.
Write $\widetilde{\pi}(\balpha)= \diag\left[ 
  \widetilde{\pi}^{(1)}, \widetilde{\pi}^{(2)}, 
\dots, \widetilde{\pi}^{(\NC)} \right] \in \mathbb{R}^{\NC \times m}$. 
Write $\one_{m_k}$ for $[1, 1, \dots, 1]' \in \mathbb{R}^{m_k \times 1}$
and $\widetilde{\one} = \diag\left[ \one_{m_1}, \one_{m_2},
\dots, \one_{m_\NC} \right]$

With $\widetilde{\pi}(\balpha)$ describing the limit behavior inside
each fast-class on the micro time scale, 
one can then consider the distribution of the exact system $X^\ep$
on the macro time scale.
Heuristically, one expects a trajectory to enter a fast-class of states
$\M_{k_1}$ and quickly iterate through many fast reactions 
until the distribution of the trajectory reaches the steady-state 
$\widetilde{\pi}^{(k_1)}$.
Eventually, a slow reaction will fire to move the trajectory to a new fast-class
$\M_{k_2}$ (see Figure~\ref{fig:TTS_CRN_example}).
Indeed, 
Writing 
\begin{align}
  \begin{aligned} \label{eq:Q-bar-matrix}
    \overline{Q} = \overline{Q}(\balpha, \bbeta)  \defd 
    \left[ \widetilde{\pi}(\balpha) \cdot \widehat{Q}(\bbeta) \cdot 
    \widetilde{\one} \right]
    \in \mathbb{R}^{\NC \times \NC} ,
  \end{aligned}
\end{align} 
we see that $\overline{Q}$ is itself a generator of 
an ``averaged'' CTMC reaction network, 
whose ``states'' correspond to fast-classes $\M_k$.
Write $\overline{X}(t)$ for the ``averaged'' process generated by $\overline{Q}$.
Together, $\overline{Q}$ and $\overline{X}(t)$ describe the limit (as $\ep \to 0$)
of the average rate that the exact process $X^\ep(t)$ moves between the 
fast-classes $\left\{ \M_k \right\}_{k=1}^\NC$ via slow reactions.

Furthermore, we can identify the elements of $\overline{Q}$ from the 
steady-state averages of the slow propensity functions.
First, note that every slow reaction carries each fast class to a unique
new fast-class; that is, if $x \leftrightarrow y$ and $\lambda_\beta(x), 
\lambda_\beta(y) >0$, then 
$r_\beta(x) \leftrightarrow r_\beta(y)$. 
Thus, $r_\beta(\M_k)$ is well-defined.
Then, using the form of $\widehat{Q}$ together with 
\eqref{eq:Q-bar-matrix}, we have
\begin{align}
  \begin{aligned}
    \label{eq:Q-bar-lambda}
    \overline{Q}_{k_1, k_2}
    &= 
     \sum_{ \substack{ \beta \in \bbeta \\ r_\beta(\M_{k_1} ) = \M_{k_2} } }
    \overline{\lambda}_\beta (\M_{k_1}; \btheta)
  \end{aligned}
\end{align}
for $k_1 \neq k_2$, and similarly we see that 
\begin{align}
  \label{eq:lambda-bar_0}
  \overline{Q}_{k_1, k_1} = 
  - \sum_{\beta \in \bbeta} \overline{\lambda}_\beta( \M_{k_1}; \btheta) 
  \defd - \overline{\lambda}_{\beta_0} (\M_{k_1} ; \btheta) .
\end{align}
With this formulation, we see that generator $\overline{Q}$ 
corresponds to a meta ``macro'' reaction network with 
the state-space $\overline{\M} = \left\{ \M_1, \M_2, \dots, \M_\NC \right\}$, 
reactions $\{ r_\beta :\beta \in \bbeta \}$ and propensities 
$\{ \overline{\lambda}_\beta(\overline{X}; \btheta) : \beta \in \bbeta \}$.
Figure ~\ref{fig:TTS_CRN_example} depicts such a macro chain
for the macro process $\overline{X}(t)$.
\begin{figure}
  \centering
    \includegraphics[width=0.45\textwidth]{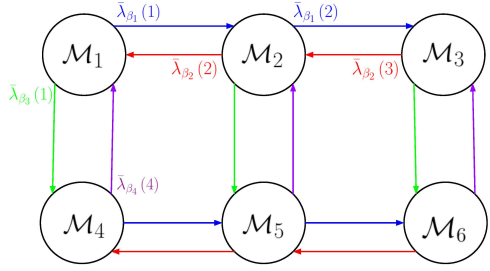}
    \caption{A depiction of the macro chain of a Two-Time-Scale Reaction Network.
      Each fast-class $\M_k$ is an equivalence class of states accessible
      by fast-reactions $r_\alpha$. Slow reactions $r_\beta$ carry each 
      fast-class to a unique new fast-class.
      The averaged system $\overline{X}(t)$ forms a meta ``macro'' Markov chain
      among fast-classes with propensities $\overline{\lambda}_\beta
      \left( \overline{X} ; \bbeta, \widetilde{\pi}(\balpha) \right)$.}
    \label{fig:TTS_CRN_example}
\end{figure}

If we can estimate the average slow-propensities
$\widehat{\overline{\lambda}}_\beta(\M_k; \btheta)$ within each fast-class
(say, through ergodic time averages of the fast-only process),
then one can simulate a trajectory of the macro process $\overline{X}(t)$
from these average propensities using {\em any} 
single-scale Monte Carlo simulation (e.g. Direct SSA, Next-Reaction, etc).
Furthermore, 
if one is ultimately concerned with estimating 
$\E_{\btheta}\left\{ f \left( X^\ep(t) \right) \right\}$
for some observable (quantity of interest) $f:\M \to \rr$, 
then one can define an augmented functional
$\overline{f}$ on $\overline{\M}$ by
\begin{align}
\label{eq:fbar_defn}
  \overline{f} \left( \M_k ; \balpha \right)
  \defd \E_{\widetilde{\pi}^{(k)}(\balpha) } \left\{ f(X) \right\} 
  =\sum_{x \in \M_k} f(x) \widetilde{\pi}^{(k)}_x(\balpha)
\end{align}
It shall be shown that 
$\E_{\btheta} \left\{ \overline{f}(\overline{X}(t)) \right\} 
\approx \E_{\btheta} \left\{ f\left( X^\ep(t) \right) \right\}$
for large enough $t$ and sufficiently small $\ep$.

To illustrate how one can implement the averaging scheme to generate 
macro-trajectories, we present the following TTS vlersion of the
Direct SSA (since it is the most succinct to write). 
In this case, the TTS SSA is essentially the same algorithm as in 
Refs. \citen{e_nested_2007}, \citen{samant_overcoming_2005}.
However, we emphasize that
the same method can be used to create a TTS version of any
exact method defined by the propensity functions.
In particular, one can just as easily construct
an analogous TTS Next-Reaction type algorithm
\cite{rathinam_efficient_2010,anderson_efficient_2012, anderson_modified_2007}
for tightly coupled trajectories.

\begin{algorithm}[TTS-SSA] \label{alg:TTS-SSA}
  \rm 
  To simulate a trajectory of the macro-process
  $\overline{X}(T)$ until macro time-horizon $T_{final}$:
  \begin{enumerate}[\bf (1)]
    \item Initialize $x$ at a macro time $T$; 
      $x \in \M_k$ for some (unknown) $k$
      
    \item Simulate the fast-only reaction network ${\widetilde{Q}^{(k) } }(\boldsymbol{\alpha})$ until time-averages
      of observable $f$ and slow propensities $\lambda_{\beta}$ relax to steady-state :
      \begin{eqnarray*}
	      \frac{1}{t} \int_0^t f\left( 
	           {\widetilde{X}^{(k)} } (s) \right)ds &\to &
              \E_{{\widetilde{\pi}^{(k)}(\balpha) } }\left\{ f(\widetilde{X}) \right\} \equiv {\overline{f}(k) } \\
       \frac{1}{t} \int_0^t \lambda_{{\beta} } \left( 
       {\widetilde{X}^{(k)} } (s);\boldsymbol{{\beta} } \right) ds 
       &\to &
        \E_{{\widetilde{\pi}^{(k)}(\balpha) }  }\left\{ \lambda_{\beta}
        \left( \widetilde{X} ; \boldsymbol{\beta} \right) \right\} \\
        &\equiv & {\overline{\lambda}_{\beta }(k; \btheta) }
     \end{eqnarray*}
   \item Observe terminal state ${\widetilde{x}^{(k)} } \sim {\widetilde{\pi}^{(k)} }$.
     Compute ${\overline{\lambda}_{\beta_0} } = 
     \sum_{j=1}^{Ms} {\overline{\lambda}_{\beta_j} }$
   \item Sample time to next slow reaction : $\Delta T \sim \Exp( {\overline{\lambda}_{\beta_0} })$
   \item Sample next slow rxn to fire 
     ${\beta^* } 
     \sim 1/ {\overline{\lambda}_{\beta_0} } \left[ {\overline{\lambda}_{\beta_1} }, \dots, 
     { \overline{\lambda}_{\beta_{Ms}} } \right] $ 
   \item Update macro time $T\leftarrow T + \Delta T$ and move to the 
   next fast class
     by taking $x= { \widetilde{x}^{(k)} } + { \zeta_{\beta^*} }$
   \item Return to {\bf (1)} until macro time horizon $T_{final}$ is reached
  \end{enumerate}
\end{algorithm}

\subsection{Convergence and Error Bounds}
Here we use the above formulation of stiff networks to establish
convergence results and error bounds for the averaged process 
obtained by Algorithm~\ref{alg:TTS-SSA}.
They are largely obtained by applying results from Ref. \citen{yin_continuous-time_2013}
to the Two-Time-Scale Markov chain developed above.
We give the statements below and defer to the
Appendix for the proofs.

From the exact chain $X^\ep(t)$, 
define a stochastic process $\overline{X^\ep}(t)$ taking values in 
$\overline{\M}=\{ \M_1, \dots, \M_\NC \}$ by 
$\overline{X^\ep}(t) = \M_k $ for $X^\ep(t) \in \M_k$. 
Note that $\overline{X^\ep}(t)$ is not, in general, a Markov chain.
However, one expects that as $\ep \to 0$, the process
$\overline{X^\ep}(t)$ converges to $\overline{X}(t)$ in some sense.

\begin{proposition}[Weak Convergence]
\label{prop:weak_convergence}
Under Assumptions~\ref{assum:recurr_states} and \ref{assum:finite_states},
 as $\ep \to 0$, 
$\overline{X^\ep}(\cdot)$ converges weakly to 
$\overline{X} (\cdot)$ in the Skorohod space 
$\mathcal{D}([0,T]; \overline{\M} )$
for any time horizon $T$.
\end{proposition}

The above proposition establishes weak convergence of the 
projection (onto fast-classes) of the exact system $X^\ep$ 
to  the averaged meta system $\overline{X}$ as $\ep \to 0$.
This is essentially the same result as established in 
Ref.~\citen{kang_separation_2013}, where the authors instead 
consider the disparity of the propensities as the system size
(molecule count) $N \to \infty$. In both formulations,
one selects a reference scale and then examines limit behavior
against the reference scale as the disparity increases 
($\ep \to 0$ or $N \to \infty$).
However, in practice one implements the averaging procedure to approximate
a system with a fixed, positive
scale disparity.
Naturally, one is then concerned about the induced error from the averaging
approximation.

Write $\overline{p}_T = \overline{p}_T(\overline{X} ; \balpha, \bbeta)$
for the probability measure (on $\overline{\M}$) induced by the averaged
process $\overline{X}$ at time $T$. 
At the end of a TTS simulation, one obtains a terminal state
$X(T)=x \sim p^0_T= p^0_T(X;\balpha, \bbeta)$, where $p^0_T$ is the
probability measure on $\M$ induced by the last state observed from the
terminal fast-class. Thus, $p^0_T$ is determined by $\overline{p}_T$ and 
$\widetilde{\pi}(\balpha)$.
Write $p^\ep_T=p^\ep(X^\ep; \balpha, \bbeta)$ for the probability measure
on $\M$ induced from the exact process $X^\ep$.
Since $p^\ep_T$ is the distribution we would see from an exact simulation,
and $p^0_T$ is the distribution from the TTS simulation, the question
becomes: What is the error of $p^0_T$ from $p^\ep_T$?
One can take a singular perturbation expansion of $p^\ep_T$ in terms
of $\ep$ and identify the leading term as $p^0_T$ to obtain the
following result.

\begin{theorem}[Error in Probability]
\label{thm:prob_err_bound}
Let $\widetilde{\kappa} = -\frac{1}{2} \max \left\{ 
\operatorname{Re}(\nu) : \nu \text{ is a non-zero eigenvalue of a }
\widetilde{Q}^{(k)} \right\} $. Then under Assumptions
\ref{assum:finite_states}, \ref{assum:recurr_states}, we have
\begin{align}
\begin{aligned}
  \label{eq:prob_err_bound}
\| p^0_T - p^\ep_T \| \le O\left( \ep
+ \exp\left\{ -\widetilde{\kappa} T/\ep \right\} \right)
\end{aligned}
\end{align}
where $\| \cdot \|$ denotes the $l_2$ norm.
\end{theorem}
In Theorem~\ref{thm:prob_err_bound}, $\widetilde{\kappa}$ is the
slowest rate of convergence of $\widetilde{Q}^{(k)}$ 
to the steady-state $\widetilde{\pi}^{(k)}$ among all
fast-classes $\M_k$. Thus, as long as the macro time horizon $T$ is
large enough to ensure the fast dynamics have relaxed to steady state
($T> -\ep/\widetilde{\kappa} \log(\ep)$),
then the error becomes $\|p^0_T - p^\ep_T \| \le O(\ep)$.

Writing $\pi^0(\balpha, \bbeta)$ for the stationary distribution 
corresponding to the TTS probability measure $p^0_T(\balpha, \bbeta)$,
it is not hard to see that $\pi^0 (\balpha, \bbeta) =
\overline{\pi} \cdot \widetilde{\pi}$,
the product of the steady-state distribution between fast-classes 
$\overline{\pi}(\balpha, \bbeta)$
and the steady-state distribution within fast-classes
$\widetilde{\pi}(\balpha)$.
Write $\pi^\ep$ for the steady-state distribution corresponding to
the exact process generated by $Q^\ep$.
Then using Theorem~\ref{thm:prob_err_bound} and exponential
convergence to the steady state, we obtain the following error bounds.

\begin{corollary}[Error in Expectation]
\label{cor:expectation-error}
Under Assumptions \ref{assum:finite_states} and
\ref{assum:recurr_states},
$\| \pi^0 - \pi^\ep \| \le O(\ep)$ and 
$\| \pi^0 - p^\ep_T \| \le O(\ep)$ 
for sufficiently large $T$. Thus, 
for all bounded functions $f$ on the state space $\M$, 
\begin{align} 
\begin{aligned}
\label{eq:error_expectation}
  \left| \E_{\overline{p}_T} \left\{ \overline{f}(\overline{X}(T)) \right\}\right.
    & - \left. \E_{p^\ep_T} \left\{ f(X^\ep(T)) \right\} \right|
   \le   \\
    & \|f\|_\infty \| p^0_T - p^\ep_T \| \le O(\ep) \\
  \left| \E_{\overline{\pi}} \left\{ \overline{f}(\overline{X}) \right\}\right.
   & - \left. \E_{\pi^\ep} \left\{ f(X^\ep) \right\} \right|
   \le   \\
   & \|f\|_\infty \| \pi^0 - \pi^\ep \| \le O(\ep)
   \end{aligned}
\end{align}
\end{corollary}

Corollary \ref{cor:expectation-error} is of great practical use,
as it says that the expected value of the macro-process $\overline{X}(T)$ 
with macro-observable $\overline{f} :\overline{\M} \to \rr$
provides an $O(\ep)$ estimate of the expected value of the exact system
$X^\ep(T)$ with observable $f : \M \to \rr$.
Since we can use TTS algorithms (such as Algorithm \ref{alg:TTS-SSA})
to quickly generate trajectories of $\overline{X}(T)$ while estimating the 
macro-observable $\overline{f}(\M_k)$ at each state along the way, this
provides a method to very quickly generate estimates of 
$\E_{p^\ep_T} \left\{ f(X^\ep(T)) \right\}$ with at most $O(\ep)$ bias.
As $\ep \to 0$, the bias decreases linearly while the computational savings
increase as $O(1/\ep)$.

\section{Two Time Scale Sensitivity Analysis}
\label{sec:TTS_sens}

Computing the system sensitivities
$S_{f,T}(\theta_i) \defd \frac{\partial}{\partial \theta_i} \E_\btheta \left\{ f(X(T)) \right\}$
with respect to reaction parameters $\theta_i \in \btheta$ 
provides great insight into the model.
As such, numerous works have 
constructed and analyzed methods to estimate
the sensitivities from sample trajectories of the system
\cite{gupta_efficient_2014, rathinam_efficient_2010, 
sheppard_pathwise_2012, wolf_hybrid_2015, wang_efficiency_2014, 
nunez_steady_2015, gupta_sensitivity_2013, 
pantazis_parametric_2013, warren_steady-state_2012}.

Different methods work better for different systems or different criteria,
but all methods have higher (sometimes stupendously higher) variance
in the estimation of 
$S_{f,T}(\theta_i)$
compared to the estimation of $ \E_{\btheta}\left\{ f(X(T)) \right\} $,
thus requiring a very large number of samples to estimate the sensitivity
precisely.
If the system is stiff (as in \eqref{eq:Q-ep_decomp})
so that each exact trajectory $X^\ep(T)$
requires a prohibitively
large computational load, then the large number of 
sample paths required to estimate the sensitivity
$S_{f,T}^\ep(\theta_i) \defd 
\partial_{\theta_i} \ \E_{p^\ep_T(\btheta)} 
\left\{ f(X^\ep(T)) \right\}$
make the problem computationally intractable.

{
Corollary~\ref{cor:expectation-error}
gives that the expectation of 
macro ``averaged'' reaction network $\overline{f}(\overline{X}(T))$
gives an accurate approximation of the expectation of
the exact network; 
$\E_{\overline{p}_T(\btheta)} 
\left\{ \overline{f} (\overline{X}(T)) \right\}
= \E_{p^\ep_T(\btheta)}
\left\{ f(X^\ep(T)) \right\} + O(\ep)$.
A natural question to ask is whether the sensitivities
of the exact system converge to the 
sensitivities of the averaged system.
Using the recent result of Ref. \citen{gupta_sensitivity_2014},
we can derive the following (the details are deferred to Appendix).

\begin{proposition}{Convergence of Sensitivities}
\label{prop:sens_converge}
\begin{multline}
\label{eq:TTS_sens_estimate}
\lim_{\ep \to 0} 
S_{f,T}^\ep(\theta_i)
=
\overline{S}_{f,T} (\theta_i) \defd
\frac{\partial}{\partial \theta_i} 
\E_{\overline{p}_T(\btheta)} 
\left\{ \overline{f} (\overline{X}(T)) \right\} 
\end{multline}
\end{proposition}

Thus, if we can compute the sensitivity of the 
macro reaction network $\overline{f}(\overline{X}(T))$ 
(whose sample paths have orders of magnitude less cost to simulate than the exact stiff network $f(X^\ep(T))$), 
then this provides an 
accurate estimate of 
the exact sensitivity.
Furthermore, since $\overline{X}(\cdot)$ is formulated as a
reaction network with propensities
$\{ \overline{\lambda}_\beta(x, \btheta) : \beta \in \bbeta \}$
and observable values $\overline{f}(\M_k)$
(both of which are estimated during a TTS simulation), 
we can apply most 
of existing single-scale 
sensitivity estimation methods to estimate 
$
\overline{S}_{f,T}(\theta_i)
$
 and thus
$
S_{f,T}^\ep(\theta_i)
$.

We note that \eqref{eq:TTS_sens_estimate} 
gives that the sensitivity of the exact system
converges to the sensitivity of the averaged system,
but does not give the rate of convergence.
Currently, this is an open question.
Since from \eqref{eq:error_expectation} we have 
the expectation converges at a rate 
$O(\ep)$, one might suspect that the sensitivity 
also converges at rate $O(\ep)$, at least for certain classes
of networks (e.g. linear propensities).
Ongoing work aims to establish the rate
of convergence via singular perturbation expansions
of sensitivity reweighting measures. 
However, the remainder of this work 
shall focus on the development and 
practical implementation of a 
multiscale Likelihood Ratio estimator of the
limit sensitivity $\overline{S}_{f,T}(\theta_i)$.

}
In what follows, we review the Likelihood Ratio
method for computing system
sensitivities for single-scale reaction networks.
Furthermore, we shall introduce a new Ergodic Likelihood Ratio
method which has much smaller variance when estimating sensitivities
at steady-state.
We then derive a Two-Time-Scale version that 
allows one to estimate the full gradient of a stiff
system using any TTS Monte Carlo method for 
simulating a macro trajectory.

\subsection{Likelihood Ratio Methods}
\label{sec:LRmethods}
Likelihood Ratio (LR) methods
\cite{plyasunov_efficient_2007, wang_efficiency_2014,  
warren_steady-state_2012, nunez_steady_2015,
glynn_likelihood_1990, mcgill_efficient_2012}
(aka the Girsanov Transform Method)
attempt to compute the derivative by reweighting the observed trajectory
by its ``score'' function of the density.
Here, one views $\btheta$ as parametrizing the probability
measure on the path-space $P(\cdot, t ; \btheta)$.
If $P(\cdot, t ; \btheta)$ is differentiable with respect to $\theta_i$,
then under mild regularity conditions we have 
\begin{multline}
\label{eq:LR_reweight}
    S_{f,t}(\theta_i) \defd
    \frac{\partial}{\partial \theta_i} \E_{\boldsymbol{\theta^0}}
    \left\{ f(X(t)) \right\} \\
    = \int_{\Omega} f(X(t,\omega)) 
    \frac{\frac{\partial}{\partial \theta_i} \big|_{\boldsymbol{\theta^0} }
    P(d\omega, t; \boldsymbol{\theta})}{P(d\omega, t ;\boldsymbol{\theta}) }
    P(d\omega, t; \boldsymbol{\theta^0}) \\
    = \E_{\boldsymbol{\theta^0}} \left\{ f(X(t)) W_{\theta_i}(t) \right\}
  \end{multline}
Using the random-time-change representation \eqref{eq:random-time-change}
it can be shown\cite{plyasunov_efficient_2007}
that the reweighting process $W_{\theta_i}(t)$ 
is a zero-mean martingale process and
can be represented by
\begin{multline}
\label{eq:W_representation}
  W_{\theta_i} (t) = \sum_{r=1}^M \int_0^t \frac{ 
    \frac{\partial \lambda_r}{\partial \theta_i} \left( X(s^-), \boldsymbol{\theta^0}
  \right) }{\lambda_r\left( X(s^-), \boldsymbol{\theta^0} \right)} dR_r(s) \\
  - \sum_{r=1}^M \int_0^t \frac{\partial \lambda_r}{\partial \theta_i} 
  \left( X(s^-), \boldsymbol{\theta^0} \right) ds ,
\end{multline}
where $dR_r(s)$ is simply the counting measure of reaction $r$ which 
equals $1$ at times $s$ at which reaction $r$ fires and is zero otherwise.
Thus, assuming one can compute 
$\partial_{\theta_i} \ \lambda_r(x,\btheta)$, then
$W_{\theta_i}(t)$ has a computationally tractable form as follows.

We write $\hat{X}_l$ for the $l$-th state of the system through a trajectory, 
and $\Delta_l$ for the holding time in the $l$-th state.
Write $T_l$ for the time of the $l$ jump, so that
$T_l = \sum_{j=0}^{l-1} \Delta_j$.
We denote $N(t)$ as the total number of reactions which have fired by time $t$
and $r^*_l$ for the index of the reaction which takes the 
system from the $l$-th state to the $l+1$-th state.
Then
$W_{\theta_i}$ has 
the explicit form

\begin{multline*}
W_{\theta_i} (t) = \\
\sum_{l=0}^{N(t)-1} \left[
    \frac{\partial }{\partial \theta_i} 
        \log{\lambda_{r^*_l}\left( \hat{X}_l, \boldsymbol{\theta^0}\right) }
  - \sum_{r=1}^M \frac{\partial \lambda_r}{\partial \theta} 
  \left( \hat{X}_l, \boldsymbol{\theta^0} \right) \Delta_l 
\right] \\
  - \sum_{r=1}^M  \frac{\partial \lambda_r}{\partial \theta} 
  \left( \hat{X}_l, \boldsymbol{\theta^0} \right) 
\left[ T - T_{N(T)} \right] .
\end{multline*}

In simulation, the LR estimate is
computed via ensemble averages 
estimated by empirical averaging
$S_{f,T}(\theta_i)
     \approx \widehat{ \operatorname{\bf LR}} (\NS,\theta_i)$
     with the empirical estimator
\begin{multline}
	\label{eq:LR-defn}
\widehat{ \operatorname{\bf LR}} (\NS,\theta_i) \defd 
    \frac{1}{\NS} \sum_{n=1}^{\NS} 
    \widehat{\left[ f(x(T)) \right]}_n
    \widehat{\left[ W_{\theta_i}(T) \right]}_n .
\end{multline}
where $\NS$ is the number sample paths,
$\widehat{\left[ f(x(T)) \right]}_n$ is the observable value at 
terminal time $T$ 
for the $n$th sample path,
and similarly $\widehat{[W_{\theta_i}(T)]}_n$ is the terminal value of
$W_{\theta_i}(T)$ for the $n$th sample path.
While the reweighting process $W_{\theta_i}(t)$ has zero mean,
its variance grows with time 
\cite{wang_efficiency_2014, 
warren_steady-state_2012}, 
making it quite inefficient for large time horizons. 
The variance can be reduced by using the centered 
likelihood ratio estimate
\begin{multline}
	\label{eq:CLR-defn}
    \widehat{\operatorname{\bf CLR}}
    (\NS, \theta_i)
     \defd
    \widehat{\operatorname{\bf LR}}(\NS, \theta_i) \\
    - \frac{1}{\NS^2} 
    \left\{ \sum_{n=1}^\NS
    \widehat{\left[ f(x(T)) \right]}_n
  \right\}
    \left\{ \sum_{n=1}^\NS
    \widehat{\left[ W_{\theta_i}(T) \right]}_n 
  \right\} .
\end{multline}
Since the $\E_{\btheta} \{ W_{\theta_i}(T) \} \equiv 0$, the
second term doesn't impose any bias into
the estimate \eqref{eq:LR_reweight}, but is coupled to
the first term to reduce the observed variance
\cite{wang_efficiency_2014}.

Suppose one is interested in the steady-state sensitivities,
$ 
S_{f,\infty}(\theta_i) \defd
\partial_{\theta_i}
\E_{\pi(\btheta)} 
\left\{ f(X) \right\}
$.
It is well known that 
$\E_{p_T(\btheta)} \left\{ f(X(T)) \right\} =
\E_{\pi(\btheta)} \left\{ f(X) \right\} + 
O\left( e^{-\kappa T} \right)$ for some mixing rate $\kappa$,  
and thus for large $T$
 one can use the terminal distribution of $f(X(T))$ and $W_{\theta_i}(T)$
in \eqref{eq:LR_reweight} to obtain an estimate of the steady-state sensitivity
with exponentially small bias \cite{warren_steady-state_2012}.
However, the major difficulty in using likelihood ratio estimates is
the large variance of the estimator $f(X(T)) W_{\theta_i}(T)$, 
which is proportional to $\Var\{ f(X(T)) \} \Var\{W_{\theta_i}(T)$ 
\cite{wang_efficiency_2014, plyasunov_efficient_2007}.
It can be seen that $\Var\left\{ W_{\theta_i }(T) \right\} = O(T)$,
so one must balance choosing a terminal time $T$ large enough to
ensure sufficient decay of the bias 
$\E_{p_T(\btheta)}\left\{ f(X(T)) \right\} - 
\E_{\pi(\btheta)} \left\{ f(X) \right\} $, yet as small as possible to 
contain the growth of the $\Var\left\{ W_{\theta_i}(T) \right\}$.
While centering as in \eqref{eq:CLR-defn} helps to reduce the variance of
the estimator, the variance is usually much larger than comparable 
finite difference of pathwise derivative methods 
\cite{wolf_hybrid_2015,wang_efficiency_2014,sheppard_pathwise_2012}.

Instead of using the terminal distribution 
$f(X(T))$ as an approximation of the steady-state distribution, 
one could instead use the ergodic-average (time-average)
$1/T \int_0^T f(X(s)) ds $.
The bias of the ergodic-average decays slower than the terminal distribution
($O(1/T)$ compared to $O(e^{-\kappa T})$), 
but has the advantage that variance decays
with time as well; that is, 
$\Var\left\{ 1/T \int_0^T f(X(s)) ds \right\} = O(1/T)$
whereas $\Var\left\{ f(X(T)) \right\} \to \Var\left\{ f(X(\infty)) \right\}
= \sigma^2$ (see Ref. \citen{asmussen_stochastic_2007} for more details).

Motivated by the variance reduction one obtains with ergodic averaging,
we introduce a new method for computing likelihood-ratio type 
steady state sensitivity estimates.
The idea is to simply replace the terminal-state observable 
$f(X(T))$ with the ergodic average $1/T \int_0^T f(X(s)) ds$ in 
the LR scheme \eqref{eq:LR_reweight} -- \eqref{eq:LR-defn}.
The philosophy is that
by incurring some small amount of additional bias in the mean value, 
the ergodic steady-state sensitivity estimate
has much smaller variance than the terminal-state distribution.
We shall refer to this method the {\em ergodic likelihood ratio},
\begin{multline}
\label{eq:ELR_defn}
    \widehat{ \operatorname{\bf ELR}} (\NS,\theta_i) \defd \\
    \frac{1}{\NS} \sum_{n=1}^{\NS} 
    \frac{1}{T}
    \widehat{\left[ \int_0^T f(x(s))ds \right]}_n
    \widehat{\left[ W_{\theta_i}(T) \right]}_n .
\end{multline}
Similarly, one can center the {\bf ELR} to derive the 
{\em centered ergodic likelihood ratio} {\bf CELR},
\begin{multline}
	\label{eq:CELR-defn}
    \widehat{\operatorname{\bf CELR}}
    (\NS, \theta_i)
     \defd
    \widehat{\operatorname{\bf ELR}}(\NS, \theta_i) \\
    - \frac{1}{\NS^2} 
    \left\{ \sum_{n=1}^\NS
    \frac{1}{T}
    \widehat{\left[ \int_0^T f(x(s))ds \right]}_n
  \right\}
    \left\{ \sum_{n=1}^\NS
    \widehat{\left[ W_{\theta_i}(T) \right]}_n 
  \right\} .
\end{multline}

In the numerical experiments, it shall be seen that the {\bf CELR } method performs much better than the {\bf CLR} for steady-state sensitivity estimation.

\subsection{TTS Likelihood Ratio}

In what follows, we describe how 
the above single-scale 
Likelihood Ratio methods
can be adapted
to the macro-process $\overline{X}(T)$
for use in
\eqref{eq:TTS_sens_estimate}.
Recall that the reaction parameters can be 
classified as fast or slow,
$\btheta=[\balpha, \bbeta]$ with
$\balpha=[\alpha_1, \dots, \alpha_{M_f} ]$ and 
$\bbeta=[\bbeta_1, \dots, \bbeta_{M_s} ]$.
To apply Likelihood Ratio methods to
compute $\partial_{\theta_i}  \E_{\btheta} 
\left\{ \overline{f}(\overline{X}(T) \right\}$,
we exploit that the macro process $\overline{X}(T)$
is identified as reaction network with propensities
\begin{multline*}
   \overline{\lambda}_{\beta_r}(\overline{X}; \balpha, \bbeta)
   = \E_{\widetilde{\pi}^{(\overline{X})}(\balpha)}
    \left\{ \lambda_{\beta_r}(X;\bbeta) \right\} \\
    = \sum_{x \in \M_{\overline{X}}}
    \lambda_{\beta_r}(x;\bbeta) 
    \widetilde{\pi}^{(\overline{X})}(x;\balpha)
\end{multline*}
(for $\beta_r \in \bbeta$), and observable
\begin{align*}
  \begin{aligned}
    \overline{f}(\overline{X}; \balpha)
    =\E_{\widetilde{\pi}^{(\overline{X})}(\balpha)}
    \left\{ f(X) \right\} 
    = \sum_{x \in \M_{\overline{X}}}
    f(x) \widetilde{\pi}^{(\overline{X})}(x;\balpha)
  \end{aligned}
\end{align*}
Thus the macro-sensitivities can be represented by
\begin{multline}
  \label{eq:TTS_LR}
  \frac{\partial}{\partial \theta_i}
  \E_{\overline{p}_T(\btheta)}
  \left\{ \overline{f}\left( \overline{X}(T);\btheta \right)  \right\} \\
  = \E_{\overline{p}_T(\btheta)} 
  \left\{ \frac{\partial}{\partial \theta_i}
  \overline{f}\left( \overline{X}(T);\btheta \right) 
  + \overline{f}\left( \overline{X};\btheta \right)
  \overline{W}_{\theta_i} (T) \right\}
\end{multline}
where the macro-reweighting process 
$\overline{W}_{\theta_i}(T)$ is given by
\begin{multline}
    \label{eq:TTS_W-rep}
    \overline{W}_{\theta_i} =
    \sum_{r=1}^{M_s} \int_0^T
    \frac{
    \frac{\partial}{\partial \theta_i} 
    \overline{\lambda}_{\beta_r}
    \left( \overline{X}(s);\btheta \right)
  }{ 
    \overline{\lambda}_{\beta_r} \left( \overline{X}(s);\btheta \right)
  } dR_{\beta_r}(s) \\
  - \sum_{r=1}^{M_s} \int_0^T
  \frac{\partial}{\partial \theta_i}
  \overline{\lambda}_{\beta_r}\left( 
  \overline{X}(s);\btheta \right) ds .
\end{multline}
Therefore, in order to apply \eqref{eq:TTS_LR}
we need to be able to compute the derivatives of the
averaged observable
$\partial_{\theta_i} \ \overline{f}\left( \overline{X}(s); \btheta \right)$
as well as the derivatives of the averaged propensity functions
$\partial_{\theta_i} \ \overline{\lambda}_{\beta_r}\left( \overline{X}(s);\btheta \right)$.

Suppose $\theta_i = \beta_i \in \bbeta$ is a slow reaction parameter.
If the original observable $f(X)$ has no direct parameter dependence,
then 
$\partial_{\beta_i} \ \overline{f}(\overline{X}(s),\btheta) \equiv 0$.
Furthermore, under mass-action kinetics, the averaged propensities 
have 
$\partial_{\beta_i} \ \overline{\lambda}_{\beta_r}(\overline{X}(s); \btheta)
= \overline{b}_{\beta_r} (\overline{X};\balpha) \delta_{i,r}$,
where $\overline{b}_{\beta_r}(\overline{X};\balpha) = 
1/\beta_r \cdot \overline{\lambda}_{\beta_r}(\overline{X};\balpha)$
is already computed during a TTS simulation 
and $\delta_{i,r}=1$ if $i=r$ and $0$ otherwise.
Thus the slow sensitivities are directly computable from
a standard TTS simulation.

Suppose $\theta_i=\alpha_i \in \balpha$ is a fast reaction parameter.
Then computing $\partial_{\alpha_i}  \overline{f}(\overline{X}(s);\balpha)$
and $\partial_{\alpha_i} \overline{\lambda}_{\beta_r}(\overline{X}(s); \balpha, \bbeta)$
is more problematic, as they only depend indirectly on $\balpha$ through the 
fast-class steady-state measures $\widetilde{\pi}(\balpha)$.
Thus explicit computation is often infeasible. 
However, one may estimate
$\partial_{\alpha_i}  
\E_{\widetilde{\pi}^{(\overline{X})}(\balpha)}
\left\{ f(X) \right\}$
and 
$\partial_{\alpha_i} \
\E_{\widetilde{\pi}^{(\overline{X})}(\balpha)}
\left\{ \lambda_{\beta_r}(X);\bbeta \right\} $
through any sensitivity analysis method from a simulation with only fast reactions.
For example, when running the fast-only simulation
(under $\widetilde{Q}^{(\overline{X})}(\balpha)$)
for equilibration in Algorithm \ref{alg:TTS-SSA}, one can compute
the corresponding likelihood ratio process 
$\widetilde{W}^{(\overline{X})}_{\alpha_i}(t)$ as in 
\eqref{eq:W_representation} 
(with $t$ large enough so that 
$\widetilde{p}^{(\overline{X})}_t(\balpha) \approx 
\widetilde{\pi}^{(\overline{X})}(\balpha)$).
Then one can estimate 
the derivitives in \eqref{eq:TTS_LR}, \eqref{eq:TTS_W-rep} by
\begin{align}
\begin{aligned}
\label{eq:micro_sensitivity}
  \frac{\partial}{\partial \alpha_i}
  \overline{f}(\overline{X}; \balpha) 
  &
 \approx 
\E_{\widetilde{p}_t^{(\overline{X})} } 
\left\{ f( \widetilde{X}(t)) 
\widetilde{W}_{\alpha_i}(t) \right\} \\
  \frac{\partial}{\partial \alpha_i}
  \overline{\lambda}_{\beta_r}(\overline{X}; \btheta) 
  &
 \approx
\E_{\widetilde{p}_t^{(\overline{X})} } 
\left\{ \lambda_{\beta_r}( \widetilde{X}(t)) 
\widetilde{W}_{\alpha_i}(t) \right\} ,
\end{aligned}
\end{align}
using the proposed CELR method \eqref{eq:CELR-defn}
during the micro-equilibration computation.
Plugging these estimated values 
into 
\eqref{eq:TTS_W-rep}
allows one to calculate
$\overline{W}_{\alpha_i}$ for each macro-trajectory,
which in turn allows for sensitivity estimation 
with respect to $\alpha_i$ in 
\eqref{eq:TTS_LR}.

We note that our derivation leads to a different form of the
multi-scale LR estimator compared 
with Ref. \citen{nunez_steady_2015}.
The latter estimated the reweighting measures for the exact process
${W^\ep}_{\alpha_i}(t)$ by adding together the
micro reweighting measures $\widetilde{W}_{\alpha_i}$
from within each fast class visited, the idea being that 
$\widetilde{W}_{\alpha_i}(t)$ is a zero-mean martingale 
which adds no new information and only increases
 in variance once the 
fast-only process has converged to steady-state.
Henceforth, we refer to this approach as the 
``Truncated Likelihood Ratio'', as it approximates the exact
reweighting coefficient $W^\ep(t)$ 
via a truncated observation
within each fast-class.
Conversely, the TTS Likelihood Ratio
uses the exact representation \eqref{eq:TTS_W-rep} for the
macro process, and then estimates the terms
via \eqref{eq:micro_sensitivity}

Lastly, we note that the above procedure
will estimate sensitivities with respect to 
the parameter set
$\btheta=[\balpha, \bbeta]$.
However, the original goal 
was to estimate sensitivities with respect to
$\btheta^\ep= [\balpha/\ep, \bbeta] 
= [\balpha^\ep, \bbeta] $.
The sensitivities of the slow parameters $\bbeta$
are the same, but the TTS Sensitivity scheme computes
fast sensitivities against the rescaled parameter $\alpha_i$
rather than against the original parameter 
$\alpha^\ep_i = \alpha_i / \ep$.
However, it can be shown
(Appendix \ref{sec:analytic_SS})
that at steady-state we have
\begin{equation}
  \frac{\partial}{\partial \alpha^\ep_i} 
  \E_{\pi^\ep} \left\{ f(X;\btheta^\ep) \right\} 
  = \ep \frac{\partial}{\partial \alpha_i}
  \E_{\pi^\ep} \left\{ f(X;\btheta^\ep)  \right\} .
  \label{eq:rescaled_sensitivities}
\end{equation}
Therefore, by multiplying the TTS sensitivity estimate (against $\alpha_i$)
by a factor of $\ep$, one thus obtains the estimate against the original
parameter $\alpha^\ep_i$. Thus one can use the TTS scheme to estimate the full gradient
of $\nabla_{\btheta^\ep} \E_{\pi^\ep} \left\{ f(X; \btheta^\ep ) \right\}$.

\section{Batch-Means Stopping Rule}
\label{sec:BMstop}
A crucial question when implementing a TTS simulation
is: How long to run the micro-equilibration for?
That is, how large a value of $t$ does one use to compute
the ergodic averages
\begin{align*}
  \overline{f}\left( \overline{X};\btheta \right)
  & \approx 
  \frac{1}{t} \int_0^t 
  f\left( \widetilde{X}^{(\overline{X})}(s);\btheta \right)ds \\
  \overline{\lambda}_\beta \left( \overline{X}; \btheta \right)
  & \approx 
  \frac{1}{t} \int_0^t
  \lambda_\beta \left( \widetilde{X}^{(\overline{X})} (s) ; \bbeta \right) ds
\end{align*}
for a desired level of accuracy $\delta$?
Taking too small a value for $t$ risks imposing a large bias.
However, the $O(1/t)$ rate of convergence for the ergodic average
implies almost nothing is gained by integrating $t$ past the relaxation
time of the system.
Furthermore, when computing the micro-sensitivities one uses
the micro-reweighting process $\widetilde{W}_{\alpha_i}(t)$
by
\begin{align*}
  \frac{\partial}{ \partial \alpha_i}
  \overline{f}(\overline{X}; \balpha)
  & \approx 
  \E_{\widetilde{p}_t^{(\overline{X})}(\balpha) } 
  \left\{ f\left( \widetilde{X}(t) \right) 
  \widetilde{W}_{\alpha_i}(t) \right\} \\
  \frac{\partial}{\partial \alpha_i}
  \overline{\lambda}_\beta
  \left( \overline{X} ; \btheta \right)
  & \approx
  \E_{\widetilde{p}_t^{(\overline{X})} (\balpha)} 
  \left\{ \lambda_\beta \left( \widetilde{X}(t) ; \bbeta \right) 
  \widetilde{W}_{\alpha_i}(t) \right\} ,
\end{align*}
where the variance of $\widetilde{W}_{\alpha_i}(t)$
increases with the time-horizon $t$.
Thus, we would ideally take the smallest value of $t$
such that 
$\| \widetilde{p}^{(\overline{X})}_t 
- \widetilde{\pi}^{(\overline{X})} \|
\le O(\delta) $.
However, different fast-classes $\overline{X}$ can have 
vastly different sizes.
This can result in significantly different relaxation times
for each class.
It is then ideal to have an 
{ \em adaptive stopping rule }
which terminates the micro (fast-only) simulations
when the ergodic averages have converged to the steady state mean.

Current implementations of an ``averaged'' or ``multi-scale''
SSA use a constant relaxation time $t_f$ 
for the micro-averaging step\cite{e_nested_2007, nunez_steady_2015}
whose choice is motivated by some a priori insight into the system.
In Refs.\citen{cao_multiscale_2005, cao_slow-scale_2005} the authors
also use a fixed time $t$, but then exploit algebraic relations
of the steady-state means to try to obtain better approximations.
In Ref. \citen{samant_overcoming_2005}, a stopping rule is developed which
determines that equilibrium is reached when 
the averaged values of the propensities
of the forward and backward reactions are
approximately equal for each reaction pair.
However, experience has shown this 
``partial-equilibrium'' stopping rule can
stop prematurely (in the transient regime)
with significant probability
for systems with relatively few reaction-pairs.
Thus, we seek to obtain a robust, adaptive stopping rule
for terminating the micro-equilibration simulation.

\subsection{Batch-Means for Steady-State Estimation}

The problem at hand is really one about 
Markov chain mixing-times
and the integrated autocorrelation time $\tau_{int}$.
Analytically, the mixing and integrated autocorrelation times
are related to the spectral gap of the underlying generator
\cite{levin_markov_2009,
geyer_practical_1992, 
kipnis_central_1986}.
Unfortunately, for large systems direct computation is usually infeasible.
Some common approaches involve estimation autocorrelation function
$A(t)$ of the process and then exploiting the relation
$\tau_{int} = 2\int_0^\infty A(t)dt$ to derive estimates
of $\tau_{int}$ from the estimates of $A(t)$
\cite{geyer_practical_1992,
berg_introduction_2004,
sokal_monte_1997}.
However, if the goal is to terminate the simulation
when the ergodic average has converged appropriately,
then these methods are indirect and can be computationally
intensive.
Another common approach is to exploit the regenerative
structure of Markov chains
\cite{meyn_markov_2009}
to obtain independent and identically distributed samples of
the process and obtain confidence bounds on the ergodic average.
However, these methods can be inefficient for complex systems
where the return time to the initial state can be quite large.

We instead turn to the method of batch means
\cite{asmussen_stochastic_2007,
alexopoulos_implementing_1996}
for determining confidence bounds 
(and thus a measure of convergence)
for the steady-state estimation problem inside each
fast-class.
The use of batch means is applicable to a wide range of problems
(any which satisfy a central limit theorem), and its 
implementation is very straightforward.
For a general Markov chain $X(s)$ with an observable function $f$,
write $Y(s)=f(X(s))$ and $\overline{Y}(t) = 1/t \int_0^t Y(s) ds$ .
We denote $f_\pi=\E_{\pi}\left\{ f(X) \right\}$ for the steady-state
value we wish to estimate.
Then under some general conditions
\cite{kipnis_central_1986}
$Y(s)$ satisfies a functional Central Limit Theorem:
\begin{align*}
  \frac{t}{\sqrt{\ep} }
  \left\{ \overline{Y}(t/\ep) - f_\pi \right\}_{t \ge 0} 
  \overset{\mathcal{D}}{\longrightarrow}
  \left\{ \sigma B(t) \right\}_{t \ge 0}
\end{align*}
in the sense of weak convergence as $\ep \to 0$.

Suppose that $t \ge N_b \tau_{relax}$, where $\tau_{relax}$ is the 
relaxation time of the system and $N_b$ is a number of ``batches''
(bins) to partition the trajectory into.
Then the batch means
\begin{align*}
  \overline{Y}_k(t) \defd 
  \frac{1}{t/N_b} \int_{(k-1) t/N_b}^{k t/N_b} Y(s) ds
\end{align*}
are approximately (as $t \to \infty$) 
independent and identically distributed samples of 
$\mathcal{N} \left( f_\pi, \sigma^2 N_b/t \right)$.
Thus 
\begin{align*}
  \sqrt{N_b} \frac{\overline{Y}(t) - f_\pi}{s_{N_b}(t)}
  \overset{\mathcal{D}}{\longrightarrow}
  T_{N_b -1}
\end{align*}
as $t \to \infty$, where $T_{N_b -1}$ is the Student's 
$t$-distribution and $ s^2_{N_b}(t)$ is the
sample variance among batches,
\begin{align*}
  s_{N_b}^2(t) \defd
  \frac{1}{N_b -1} \sum_{k=1}^{N_b}
  \left[ \overline{Y}_k(t) - \overline{Y}(t) \right]^2 .
\end{align*}
Thus, for $t$ sufficiently large, 
a $(1-\delta_{CI})100\%$ confidence interval 
for the value of $f_\pi$ is given by
$\overline{Y}(t) \ \pm \ MOE \left( t,N_b, \delta_{CI} \right) $,
where
\begin{multline}
  \label{eq:batch-means_CI}
 MOE \left( t,N_b, \delta_{CI} \right) \defd
 \left( t_{quantile} \right)
  \frac{s_{N_b}(t)}{\sqrt{N_b}} \\
\end{multline}
and $t_{quantile}$ is the $(1-\delta_{CI}/2)$th quantile of
the Student's $t$-distribution with $N_b -1$ degrees of freedom.

The usual perspective for applying batch means
is that one has a fixed set of data 
$\left\{ Y(s): s \in [0,t] \right\}$
to partition, and then must choose the number of
batches $N_b$ appropriately so that 
each batch length $t/N_b$ is long enough
so that the batch mean errors 
$\left[ \overline{Y}_k(t) - \overline{Y}(t) \right]$
are approximately independent, identically distributed, and Gaussian.
One then often chooses $N_b$ to be relatively small 
(say, 5 to 30) 
\cite{geyer_practical_1992,glynn_simulation_1990}
to ensure the independent and Gaussian assumptions hold.
When viewing the asymptotic structure as the amount of data $t$ grows,
then one can ensure that the asymptotic central limit theorem
holds if the number of batches grows as $N_b(t) \simeq \sqrt{t}$.
In Ref. \citen{alexopoulos_implementing_1996}, the authors consider
strategies which let the number of batches grow if the correlation
between batches is near 0, and otherwise hold $N_b(t)$ fixed until
the batch correlation decays to 0.

Since our goal is to simulate only enough (micro-scale) data
so as to determine the steady-state values $f_\pi$, we instead
take the perspective that one has a fixed number of
batches $N_b$ desired, and that one should generate data 
$\left\{ Y(s) : s \in [0,t] \right\} $ until each
of the batch means $\overline{Y}_k(t)$ are (approximately) 
independent and identically distributed about $f_\pi$.
For a fixed level of precision $\delta_{precise}$,
confidence level $\delta_{CI}$, 
and the number of batches (independent samples)
$N_b$, the Batch-Means Stopping Rule terminates the 
simulation when $MOE(t)=MOE(t,N_b, \delta_{CI}) \le \delta_{precise}$,
where $MOE(t)$ is defined by \eqref{eq:batch-means_CI}.
Figure \ref{fig:Batching_Diagram} gives a depiction of how the 
Batch-Means Stopping Rule is implemented.

\begin{figure}
  \centering
    \includegraphics[scale=.35]{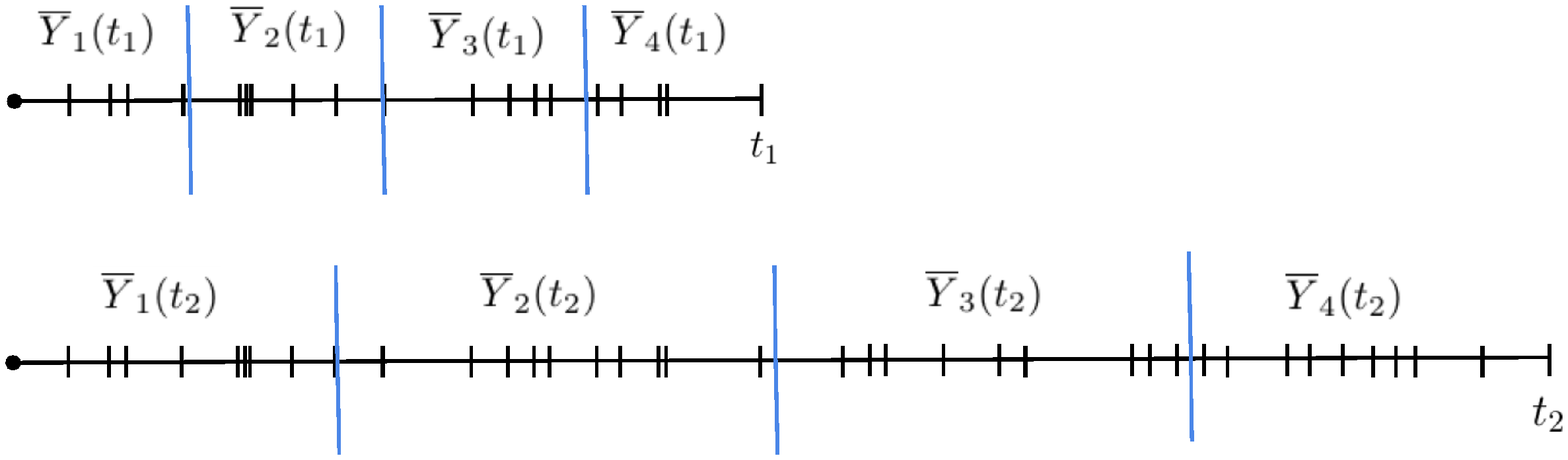}
    \caption{A sketch of the batch-means stopping rule.
      The process is simulated for a fixed number of jumps ($N_J=20$) to a terminal
      time $t_1$,
      and then the trajectory is partitioned into $N_b=4$ batches to compute
      the variance between the batch means $\overline{Y}_k(t_1)$.
      If the confidence bounds are precise enough 
      ($MOE(t_1)\le \delta_{precise}$), then the simulation is terminated
      and each batch gives an iid sample of 
      $\mathcal{N}\left( f_\pi, \sigma^2 \right)$.
    Otherwise, another $N_J=20$ jumps are simulated and the process is repeated.}
    \label{fig:Batching_Diagram}
\end{figure}

In addition to giving an on-line estimate of the relaxation time of the system,
the Batch-Means Stopping Rule gives $N_b-1$ (nearly) independent samples
of trajectories with initial distribution approximately equal to the stationary 
distribution. Furthermore, one can compute the reweighting coefficients 
$W_{k,\theta}(t)$ in each batch to give $N_b$ (nearly) independent samples
of the steady-state reweighting coefficients (in a manner similar to 
the ``Time-Averaged Correlation Function'' 
method of Ref. \citen{warren_steady-state_2012}).
 
\subsection{Batch-Means Stopping Implementation}
\label{sec:BMSI}

Suppose we have a general reaction network with 
$M_r$ reactions and $M_\theta$ reaction parameters. 
Here we allow the possibility
that $M_r \neq M_\theta$ for general propensity functions $\lambda_r(x;\btheta)$
(e.g. Michaelis-Menten kinetics), 
whose parameter derivatives $\partial_{\theta_i} \lambda_r(x;\btheta)$
we can compute explicitly for all $i=1,\dots,M_\theta$ and all $r=1, \dots, M_r$.
Denote by $\zeta_r$ the stoichiometric vector for the $r$th reaction.
Our goal is to estimate the gradient
$\nabla_\theta \E_{\pi(\btheta)}\left\{ f(X) \right\}$ 
for some observable function $f$.

We introduce the following notation for the batch-means stopping rule.
$\hat{X}(n)$ is the $n$th state of the reaction network,
$T(n)$ is the time of the $n$th jump and
$\hat{f}(n)=f(\hat{X}(n))$ for the value of the observable at the $n$th state.
$F(n)= \int_0^{T(n)} f\left( X(s) \right)ds$ is the 
time-integrated value of $f$ up to the $n$th jump,
$r^*(n)$ is the reaction which fires at jump $n$ 
(taking the system from $\hat{X}(n)$ to $\hat{X}(n+1)$),
$\hat{e}_i$ is the vector in $\rr^{1 \times M_\theta}$
with $1$ in the $i$th component and zeros elsewhere.
$R(n, \btheta) \in \rr^{1 \times M_\theta}$ and 
$B(n,\btheta) \in \rr^{1 \times M_\theta}$ are the 
first and second terms of \eqref{eq:W_representation} 
with respect to each of the parameters $\theta_i \in \btheta$
($i=1, \dots, M_\theta$).
$N_b$ is the number of batches (approximately independent samples) to be used,
$\delta_{CI}$ is the desired confidence level (for a $(1-\delta_{CI})100 \%$
confidence interval, and 
$\delta_{precise}$ is the maximum allowed radius of the confidence interval
at the stopping time.
$N_J$ is the number of jumps to simulate before retesting the batches
for convergence.
Then one can write the batch-means stopping rule as follows.

\begin{algorithm}[Batch-Means Stopping Rule with Sensitivity Estimation]
  \label{alg:batch-means} \rm ~\\

  \begin{enumerate}[\bf (1)]
    \item {\bf Initialize }
      \begin{itemize}
	\item $\hat{X}(0)= x_0$,
	  $\hat{f}(0) = f\left( \hat{X}(0) \right)$, 
	  $T(0)=0$,
	  $F(0)=0$, 
	  $R(0)= [0,\dots,0] \in \rr^{1 \times M_\theta}$,
	  $B(0)= [0, \dots, 0] \in \rr^{1 \times M_\theta}$.
	  ${ tests}=0$ (number of times the data has been tested
	  for convergence).
	  Calculate $t_{quantile}=(1-\delta_{CI}/2)$ quantile of a
	  Student's $t$-distribution with $N_b-1$ degrees of freedom.

      \end{itemize}

    \item {\bf Generate and Record Data}
      Simulate $N_J$ jumps and record values immediately after each jump.
	For $n=N_J \cdot {tests}, \dots, N_J \cdot ( {tests}+1) -1$, 
	  \begin{itemize}
	    \item Compute $\lambda_r(\hat{X}(n);\btheta)$, 
	      $\frac{\partial}{\partial \theta_i} \lambda_r(\hat{X}(n), \btheta)$
	      for all $r=1, \dots, M_r$ and all $\theta_i =1, \dots M_\theta$.
	      Set $\lambda_0(\hat{X}(n); \btheta)
	      = \sum_{r=1}^{M_r} \lambda_r (\hat{X}(n), \btheta)$, and 
	      $\frac{\partial}{\partial \theta_i} \lambda_0(\hat{X}(n), \btheta)
	      = \sum_{r=1}^{M_r} \frac{\partial}{\partial \theta_i}
	      \lambda_r(\hat{X}(n), \btheta)$ for all $\theta_i \in \btheta$.
	    \item Sample $\Delta t(n) \sim \Exp\left( 
	      \lambda_0(\hat{X}(n),\btheta) \right)$, and 
	      $r^*(n) \sim  
	      \frac{1}{\lambda_0(\hat{X}(n), \btheta)}$ \\ 
	       \qquad $\qquad \times
	      \left[ \lambda_1(\hat{X}(n);\btheta), \dots, 
	      \lambda_{M_r}(\hat{X}(n);\btheta) \right] $.
	    \item Update
	      $T(n+1)= T(n) + \Delta t(n) $ \\
	      $\hat{X}(n+1)= \hat{X}(n)+ \zeta_{r^*(n)}$\\
	      $\hat{f}(n+1)= f(\hat{X}(n+1))$ \\
	      $F(n+1) = F(n) + \hat{f}(n) \cdot \Delta t(n)$ \\
	      $R(n+1,\btheta)= R(n,\btheta)
	      + \sum_{i=1}^{M_\theta} 
	      \hat{e}_i   \\  \qquad \times \ 
	      \left[ \frac{\partial}{\partial \theta_i}
	      \lambda_{r^*(n)}\left( \hat{X}(n), \btheta \right) \right]
	      / \lambda_{r^*(n)}\left( \hat{X}(n), \btheta \right)$ \\
	      $\hat{b}(n, \btheta)=
	      \sum_{i=1}^{M_\theta} \hat{e}_i \cdot
	       \left[ \sum_{r=1}^{M_r} \frac{\partial}{\partial \theta_i}
	      \lambda_r (\hat{X}(n), \btheta) \right] $ \\
	      $B(n+1,\btheta) = B(n, \btheta) + \hat{b}(n, \btheta) \cdot
	      \Delta t(n)$ \\
	      $W(n+1,\btheta) = R(n+1, \btheta) - B(n+1, \btheta) $.

	      
	  \end{itemize}

    \item{ \bf Test Batches for Convergence}
      \begin{itemize}
	\item $N_{end} = N_J \cdot ({tests}+1)$= index of last available 
	  data point. \\
	  $\bar{Y}= F\left( N_{end} \right)/ T(N_{end})$ = total time-averaged value \\
	  $t_{batch}=T(N_{end})/N_b$= time length each of batch \\
	  Initialize $F^B=0$ (total integral through end of previous batch).
	\item for $ k=1, \dots, N_b$
	  \begin{itemize}
	    \item $ind_B(k)$= $\max \left\{ n : 
		T(n) \le k\cdot t_{batch} \right\}$=index of the last jump
		in batch $k$.
	      \item $F_A^B(k) = F(ind_B(k)) + \hat{f}(ind_B(k)) $ \\
		$ \times \left[ k\cdot t_{batch} - T(ind_B(k)) \right]
		- F^B$=total integrated value of $f(X(s))$ inside batch $k$.
	      \item $\overline{Y}_k =F_A^B(k)/t_{batch}$= $k$th batch-mean
	      \item $F^B \leftarrow F^B + F_A^B(k) $ 
		(update integral to end of previous batch)
	  \end{itemize}
	
	\item $s_{N_b}^2= \frac{1}{N_b-1 } \sum_{k=1}^{N_b}
	  \left[ \overline{Y}_k - \overline{Y} \right]^2$
	  = variance between batches
	\item $MOE= t_{quantile} * \sqrt{ s^2_{N_b} / N_b}$
	  = margin of error for confidence interval
	\item If $MOE \le \delta_{precise}$, then go to {\bf (4)}.
	  Else, ${tests} \leftarrow {tests}+1$ and go back to {\bf (2)}.
      \end{itemize}

    \item{ \bf Compute LR Weights} in each batch.
      
	Initialize $W^B(\btheta)=[0,\dots,0] \in \rr^{1 \times M_\theta}$.
	For $ k=1, \dots, N_b$, 
	  \begin{itemize}
	    \item $W_A^B(k, \btheta) 
	      = W(ind_B(k), \btheta)$  \\ $
		- \hat{b}(ind_B(k), \btheta) 
		\cdot \left[ k\cdot t_{batch} - T(ind_B(k)) \right] $ \\ $
		- W^B(\btheta) $
	      \item $W^B(\btheta) \leftarrow W^B(\btheta) + W_A^B(k,\btheta) $
	  \end{itemize}

	\item{\bf Compute Sensitivity Estimates}
	  for  $\nabla_\btheta \E_{\pi(\btheta)} \left\{ f(X) \right\}$:
	  \begin{itemize}
	    \item{Likelihood Ratio}
	      \begin{align*}
		{\bf LR} = \frac{1}{N_b}
		\sum_{k=1}^{N_b} \hat{f}\left( ind_B(k) \right)
		W_A^B(k, \btheta)
	      \end{align*}

	    \item{Centered Likelihood Ratio}
	      \begin{multline*}
		{\bf CLR} = {\bf LR} \\
		- \left[ \frac{1}{N_b} \sum_{k=1}^{N_b} 
		\hat{f}\left( ind_B(k) \right)\right] 
		\cdot \left[ \frac{1}{N_b} \sum_{k=1}^{N_b}
		  W_A^B(k, \btheta) \right]
	      \end{multline*}
\item{Ergodic Likelihood Ratio} \begin{align*} {\bf ELR} = \frac{1}{N_b} \sum_{k=1}^{N_b} \overline{Y}_k W_A^B(k,\btheta)
	      \end{align*}

	    \item{Centered Ergodic Likelihood Ratio}
	      \begin{align*}
		{\bf CELR} = {\bf ELR} 
		- \left[ \frac{1}{N_b} \sum_{k=1}^{N_b} 
		\overline{Y}_k\right] \cdot
		\left[ \frac{1}{N_b} \sum_{k=1}^{N_b} 
		W_A^B(k,\btheta) \right]
	      \end{align*}
      \end{itemize}
      
  \end{enumerate}
  
\end{algorithm}


\section{Simulation Results}
\label{sec:sim_results}
Here we present numerical results to 
display the performance of the proposed algorithms. 
In what follows, 
we compare the output of an exact simulation 
at the single time scale (STS)
to the accelerated two time scale (TTS) approximation.
From \eqref{eq:TTS_sens_estimate}, 
we expect the differences in observable 
averages and their derivatives to be $O(\ep)$. 
Because differences are small, 
we use a simple test system for which many samples can be run to obtain
accurate statistics. 

Consider a reaction network with species A, B, and C and isomerization reactions given by
\[
  \left.
  \begin{array}{rcl}
    A  \overset{k_1 / \ep}{\rightarrow} B , & \quad
    B  \overset{k_2 / \ep}{\rightarrow} A  , & \quad 
    B  \overset{k_3}{\rightarrow}  C
  \end{array}
  \right.
\]
For small values of $\ep$, the system becomes stiff as the isomerization between A and B reaches equilibrium much faster than B is converted to C. A TTS approximation assumes that $A\rightarrow B$ and $B\rightarrow A$ are fast and equilibrated.

{We first compare the output of the accelerated TTS simulation
against the exact STS simulation for varying levels of stiffness $\ep$}.
For our simulations, initial conditions of 
$(A_0,B_0,C_0)=(100,0,0)$ and the parameters $(k_1,k_2,k_3)=(1,1.5,2)$ 
are chosen.
10,000 replicate (independent) trajectories 
are run for various values of $\ep$. 
Statistics are taken at a termination time of $t=0.5$s. 
Species averages are calculated as arithmetic averages 
over the independent trajectories while sensitivities 
are computed with the CLR method shown in \eqref{eq:CLR-defn}. 
The error due to statistical averaging is estimated using t-test statistics for
averages and a bootstrapping method for sensitivities. 
Sensitivities with respect to the ``slow'' parameter $k_3$ are displayed
for each species.
As discussed in previous works\cite{nunez_steady_2015},
sensitivities with respect to parameters related to fast reactions encounter 
significant noise, 
and thus we omit them in order to clearly observe 
the difference between STS and TTS.  
Figure \ref{fig:STS_eps_converge} shows the disparity between the STS and 
TTS systems for various values of $\ep$. Errors are normalized by the TTS value such that $\mathrm{Error}=\frac{\mathrm{STS}-\mathrm{TTS}}{|\mathrm{TTS}|}$. 
{Indeed, one observes the difference is proportional to $\ep$,
as expected from Corollary~\ref{cor:expectation-error} and 
\eqref{eq:TTS_sens_estimate}. }

\begin{figure}
  \centering
    \includegraphics[width=0.45\textwidth]{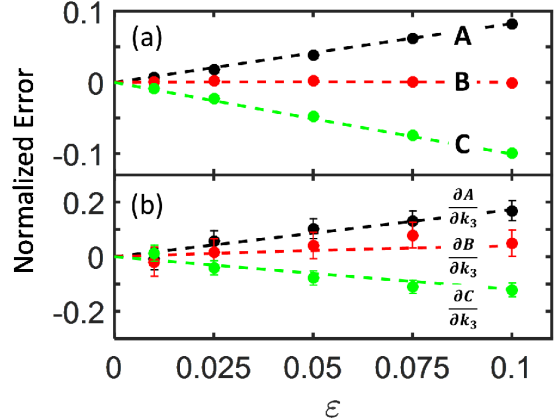}
    \caption{Normalized error of the two-time-scale (TTS) (accelerated, approximate) simulation from the 
    exact single-time-scale (STS) simulation. 
    The plot (a) shows errors in species averages while the
    plot (b) shows errors in sensitivities of each species with
    respect to $k_3$. 
    {Points indicate simulation statistics while dashed lines confirm the linear trend.} 
     The error bars are 95\% confidence intervals of statistical noise. 
     }
    \label{fig:STS_eps_converge}
\end{figure}


\begin{figure}
  \centering
     \includegraphics[width=0.45\textwidth]{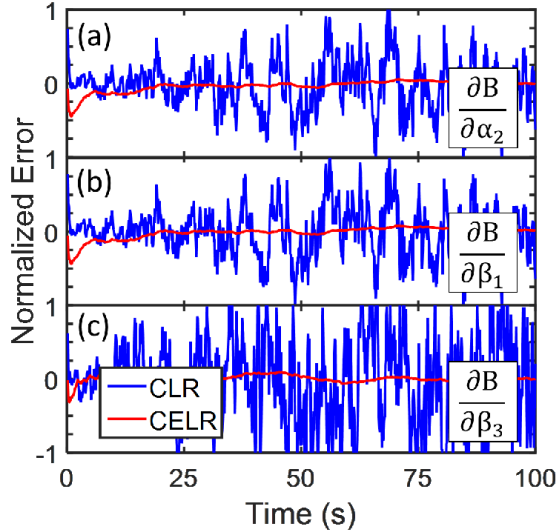}
    \caption{Time evolution of normalized errors obtained for sensitivities computed by CLR and CELR estimators. 
     Simulation estimates are referenced and normalized by the analytical solution at each time point to compute the error. 
    The graphs (a), (b), and (c) show estimates of different sensitivity indices for the species B. 
    {One observes that the variance of the CLR estimates 
    increases with time, making for inefficient estimation.
    In contrast, the CELR estimates (the solid red line) converge quickly to the
    steady state sensitivities with variance which is roughly constant with time, 
    allowing for efficient steady-state sensitivity estimation.}
    }
    \label{fig:TTS_ergodicSAerror}
\end{figure}

Next, the CLR and CELR methods from Section \ref{sec:LRmethods} are tested in performing sensitivity analysis in a TTS system. The reaction network described in Section \ref{sec:TTSRN} is simulated using Algorithm \ref{alg:TTS-SSA}. To assess convergence of the microscale distribution, the batch-means stopping criterion described in Section \ref{sec:BMstop} is used with a tolerance of $\delta=0.05$. 1000 replicate trajectories are run to a time horizon of $t_f=100$s. The initial conditions used are $(A_0,B_0,*_0)=(30,60,10)$ and the parameters used are 
$(\alpha^\ep_1, \alpha^\ep_2, \beta_1,\beta_2, \beta_3) =
(1/\ep,1.5/\ep, 2, 1, 0.4)$. 

  Species populations, time-averaged species populations, and trajectory derivatives are recorded for each run over 
 time. Using these recorded statistics, the sensitivities for all 15 (3 species and 5 parameters) 
 species/parameter combinations are computed at each time point. Figure \ref{fig:TTS_ergodicSAerror} 
 shows the time evolution of the normalized errors in sensitivity estimates {of the species B} over time. Estimated values from simulation are referenced 
 to the analytical answer as computed from a differential algebraic equation 
 (see Appendix \ref{sec:AnalylicSolnExample}) and normalized by that amount so that $\mathrm{Error} = \frac{\mathrm{estimated} - \mathrm{analytical}}{|\mathrm{analytical}|}$.

As expected, CLR estimates are noisy, with variance that grows linearly with time.
 At short times, the variance is small enough to obtain reasonable estimates. 
As time increases, the noise becomes significant with respect to the actual 
values (the magnitude of the normalized error becomes comparable to 1). 
In contrast, the ergodic likelihood ratio (CELR) fails at short times with a noticeable bias. 
However, the bias, which exists due to a relaxation period, decays as $O(1/t)$ when time increases and the system approaches its steady-state. 
The variance of the CELR estimates remain constant because the variance of trajectory 
derivatives increases linearly in time while the variance of ergodic species 
averages is proportional to $1/t$. 
At long times 
(where the CLR is too noisy for efficient estimation), 
the ergodic likelihood ratio obtains accurate estimates with very low variance. 
Therefore, it is advisable to use the CLR method for 
short times (in the transient regime) and the CELR method 
for long times to obtain steady state values.

{Table \ref{tab:table1} shows the error and statistical noise 
of the CLR and CELR estimations of sensitivities of the species B at a short 
($t=1.3$s) and long ($t=100$s) times. 
Statistical noise is obtained from bootstrapping the samples used to compute the sensitivity estimates. 
At $t=1.3$, the CLR method has low error (theoretically, there is no bias) 
as well as low variance. 
The CELR estimator has a similarly low variance, but high error (due to the $O(1/t) $ bias).
At $t=100$s, the CLR estimates have much higher variance which induces
large empirical error. 
In contrast, the bias of the CELR estimate decreases in time while the variance remains low,
providing very small empirical errors at large times. 
}

\begin{table}[h!]
  \begin{center}
    \caption{Comparison of the error and statistical noise for the CLR and CELR estimators 
    at a short time (during the transient regime) and at a long time (corresponding to steady state). 
     Values in the table refer to the sensitivity of the species B with respect to the parameter given by the row label. Values are reported as a percent of the analytically obtained sensitivity value.}
    \label{tab:table1}
    \begin{tabular}{ccc|cc}
    \hline
    & \multicolumn{4}{c}{Percent Error} \\
    \hline 
    & CLR & CELR & CLR & CELR\\
    & \multicolumn{2}{c|}{$t=1.3$s} & \multicolumn{2}{c}{$t=100$s}  \\
    \hline 
    $\alpha_1$ & -1.4 & -40.0 & -83.5 & 0.0 \\
    $\alpha_2$ & -3.0	& -40.3 & -63.0 & -1.1 \\
    $\beta_1$ & 0.2 & -39.1 & -64.1 & 1.2 \\
    $\beta_2$ & -1.1 & -16.7 & -22.6 & -2.8 \\
    $\beta_3$ & -1.4 & -23.3 & 20.6 & -0.1 \\
    \hline 
    & \multicolumn{4}{c}{Half-length of 95\%} \\
    & \multicolumn{4}{c}{Confidence Interval} \\
    \hline 
    $\alpha_1$ & 12 & 8 & 99 & 11 \\
    $\alpha_2$ & 12 & 7 & 94 & 11 \\
    $\beta_1$ & 12 & 7 & 91 & 11 \\
    $\beta_2$ & 13 & 10 & 108 & 12 \\
    $\beta_3$ & 21 & 14 & 171 & 18  
    \end{tabular}
  \end{center}
\end{table}


\section{Conclusions}
\label{sec:conclusion}
This work develops a Two-Time-Scale (TTS) framework for multiscale reaction networks.
By decomposing the system into ``fast-classes'', one can approximate the behavior 
of the multiscale system by a lower-dimensional, single-scale, ``macro-averaged'' reaction network.
By applying a singular perturbation expansion of the underlying probability measures,
we have established rigorous bounds on the 
bias induced by the approximate macro-averaged model.
We then proposed a TTS algorithm for simulating the macro reaction network, 
using an adaptive batch-means stopping rule for 
determining when the micro-scale dynamics have sufficiently equilibriated.

In addition, we have shown that the sensitivities of the macro-averaged system
provide accurate approximations for the multiscale system.
Since the macro-averaged system is single-scale, it is possible to incorporate most existing sensitivity
estimation methods to the TTS algorithm to obtain estimates of the system sensitivities.
We proposed an Ergodic Likelihood Ratio estimator for steady-state sensitivity analysis, 
and demonstrated how it can be adapted to the Two-Time-Scale algorithm.
A simulation was then used to confirm the analytic error bounds and demonstrate the efficiency of the 
TTS Ergodic Likelihood Ratio estimator.


\section{Acknowledgements}

This work is supported by the U.S. Department of Energy Office of Science, Office of Advanced Scientific Computing Research, and Applied Mathematics program under Award No. DE-SC0010549.

\appendix
\section{Analytic Stationary Sensitivities}
\label{sec:analytic_SS}
For ergodic systems whose state space is relatively small
(such that the generator $Q(\btheta)$ can be explicitly constructed),
one can compute the steady-state probability vector $\pi(\btheta)$ 
by solving the linear system 
$\pi \cdot Q =\boldsymbol{0}$ 
and $\pi \cdot \boldsymbol{1} = 1$.
Here, we show that one can express the 
sensitivities of the steady-state measure
$\frac{\partial}{\partial \theta} \pi(\btheta)$ 
explicitly as a linear transformation of the nominal measure
$\pi(\btheta)$.
The representation we shall construct 
is an adaptation of the discrete-time technique in
Ref. \citen{funderlic_sensitivity_1986}
to continuous-time Markov chains,
and exploits the algebraic properties of the 
pseudo-inverse $Q^+$ of $Q$.

By differentiating $\pi(\btheta) \cdot Q(\btheta) = \boldsymbol{0}$,
 we have the relation
\[ \frac{\partial \pi}{\partial \theta} Q = -\pi \frac{\partial Q}{\partial \theta} \]
Then by 
expanding
with the Moore-Penrose pseudo-inverse,
we have
\[
  \left( \frac{\partial \pi}{\partial \theta} \right)' 
  = \left( Q \right)'\left( \frac{-\partial Q}{\partial \theta} \right)' \pi'
  +\left[ I - \left( Q' \right)^{+} Q' \right] w
\]
for some vector $w$.

Now, we see by 
the projection property
of $(Q')^+Q' = (Q Q^+)' = Q Q^+$ that the operator
$I-QQ^+$ is the projection operator onto the kernel of $Q'$ = span of $\pi'$,
so that $(I-QQ^+)w = \gamma \pi'$ for some scalar $\gamma$.
Thus we have the relation
\[
  \frac{\partial \pi}{\partial \theta} 
  = \pi \left( \frac{-\partial Q}{ \partial \theta} \right)
  Q^+ + \gamma \pi
\]
for some $\gamma \in \mathbb{R}$. 
It remains to determine $\gamma$ to have a 
method of relating the sensitivity coefficient 
to a linear transformation of $\pi$.

Now, we can see that $\pi(\theta) \cdot \boldsymbol{1} = 1$, so that
$\frac{\partial \pi}{ \partial \theta} \cdot \boldsymbol{1} = 0$.
Thus, we have
\begin{align*}
  0 = \frac{\partial \pi}{\partial \theta} \cdot \boldsymbol{1}
  = \pi \left( \frac{- \partial Q}{ \partial \theta} \right) Q^+ 
  \boldsymbol{1} + \gamma \pi \boldsymbol{1}
\end{align*}
so that 
\[
  \gamma = \gamma \pi \boldsymbol{1} = \pi \frac{\partial Q}{ \partial \theta} 
  Q^+ \boldsymbol{1} .
\]
Putting the above together, we can write $\frac{\partial \pi}{ \partial \theta} $ as
\begin{equation}
  \label{eq:pseudo-inverse_sensitivity}
  \frac{\partial \pi}{ \partial \theta} =
  \pi \left( \frac{\partial Q}{ \partial \theta} \right) Q^+
  \left[ \boldsymbol{1} \pi - I  \right].
\end{equation}

For reaction networks with relatively small state space $\M$,
\eqref{eq:pseudo-inverse_sensitivity} can provide a 
tractable method of computing the sensitivities of the steady-state measure
$\frac{\partial \pi}{\partial \theta}$ 
and thus of the steady-state expected value
\begin{multline}
  \label{eq:analytic_SS_sens}
  \frac{\partial}{\partial \theta} 
  \E_{\pi(\btheta)} \left\{ f(X ; \btheta ) \right\} \\
  = \sum_{x \in \M} \frac{\partial f(x;\btheta)}{\partial \theta} \pi(x; \btheta) 
  + \sum_{x \in \M} f(x;\btheta) \frac{\partial \pi(x;\btheta)}{\partial \theta} .
\end{multline}
For example, for 
reaction networks with mass-action propensities,
the entries of $Q$ shall always be linear in the parameters $\theta$.
Therefore, the derivatives $\frac{\partial Q}{\partial \theta}$ 
are easy to compute (whereas the stationary distribution $\pi(\btheta)$ can be 
quite complex as a function of $\theta$).
Thus by computing the pseudo-inverse $Q^+$ from $Q$ at the nominal value of 
$\btheta$ (e.g., via singular value decomposition),
one can analytically compute the system sensitivities without need of Monte Carlo
simulation.
For larger reaction networks, it may be possible to combine the Finite State 
Projection method\cite{munsky_finite_2006,sunkara_optimal_2010}
with this pseudo-inverse technique to estimate the exact sensitivities
by computing the analytic sensitivities of the reduced system.

Lastly, we use this representation to show that 
\eqref{eq:rescaled_sensitivities} holds.
Write the exact generator $Q^\ep$ as
$Q^\ep = Q(\btheta^\ep) = (1/\ep) \widetilde{Q} (\balpha) + \widehat{Q}(\bbeta)
= \widetilde{Q}(\balpha^\ep) + \widehat{Q}(\bbeta) $.
For a fast reaction parameter $\alpha^\ep_i = \alpha_i/\ep$ 
a chain rule gives
$ { \partial_{\alpha_i} }  \lambda(x; \balpha^\ep) 
= \left[ {\partial_{\alpha^\ep_i} } 
\lambda(x; \balpha^\ep ) \right] /\ep$, 
from which it follows (using \eqref{eq:Q-ep_decomp})  that
$\ep \left[ {\partial \alpha_i } \ \widetilde{Q}(\balpha^\ep) \right]
  = {\partial_{\alpha_i }}  \widetilde{Q}(\balpha)
  ={\partial_{\alpha^\ep_i } }  \widetilde{Q}(\balpha^\ep) $.
Putting this relation into 
\eqref{eq:pseudo-inverse_sensitivity}, 
we then have
\begin{equation}
  \label{eq:rescaled_SS_sens_pi}
  \frac{\partial}{\partial \alpha^\ep_i} \pi^\ep 
  = \ep \left[ \frac{\partial}{\partial \alpha_i} \pi^\ep \right].
\end{equation}
Similarly we see that
$\partial \alpha_i \ f(x, \balpha^\ep)
= \left[ \partial_{\alpha^\ep_i} \ f(x, \balpha^\ep) \right] /\ep$.
Finally, putting these relations into 
\eqref{eq:rescaled_SS_sens_pi}
we obtain \eqref{eq:rescaled_sensitivities},
\begin{align*}
  \frac{\partial}{\partial \alpha^\ep_i} 
  \E_{\pi^\ep} \left\{ f(X;\balpha^\ep, \bbeta ) \right\}
  =
  \ep \frac{\partial}{\partial \alpha_i} 
  \E_{\pi^\ep} \left\{ f(X;\balpha^\ep, \bbeta ) \right\} .
\end{align*}

\section{Proofs of Results}
\label{sec:proofs}
In this section we outline the proofs on the error bounds of 
the averaged reaction network and convergence of the sensitivities.
The error bounds in this work are
largely direct applications of the results in Ref.
\citen{yin_continuous-time_2013}
to the Two-Time-Scale reaction networks formulated here.
We present an overview of the proofs for insight and completeness.
Similarly, the sensitivity convergence result comes from
Ref. \citen{gupta_sensitivity_2014}; we shall only how to fit their result to the Two-Time-Scale framework.

\subsection{Proof of Theorem~\ref{thm:prob_err_bound}:} Using the formulation of the exact generator
$Q^\ep(\btheta) = (1/\ep)  \widetilde{Q}(\balpha) + \widehat{Q}(\bbeta)$,
the error bound on the induced probability measures of the exact and 
averaged systems is a direct application of Theorem 4.29 of 
Ref. \citen{yin_continuous-time_2013}.
We outline the main steps below.

Write $p^\ep(t) \in \rr^{1 \times m}$ for the probability measure 
of the exact system at time $t$. From the Kolmogorov forward equation
(a.k.a. Chemical Master Equation), we have
\begin{equation}
  \frac{d p^\ep(t)}{dt} = p^\ep(t) 
  \left[ \frac{1}{\ep} \widetilde{Q}(\balpha) 
  + \widehat{Q}(\bbeta) \right] .
  \label{eq:CME}
\end{equation}
Define the differential operator $\mathcal{L}^\ep$ 
on functions with values in $\mathbb{R}^{1 \times m}$ by
$\mathcal{L}^\ep f = \ep \frac{d f}{dt} 
- f (\widetilde{Q} + \ep \widehat{Q}) $.
Then $\mathcal{L}^\ep f =0$ if and only if $f$ solves
the CME \eqref{eq:CME}.
The form of the differential equation \eqref{eq:CME} 
suggests the plausibility of a singular perturbation expansion
of $p^\ep(t)$ by
\begin{equation}
  \label{eq:SP-expansion}
  p^\ep(t) = 
  \sum_{i=0}^\infty \ep^i \phi_i(t) 
  + \sum_{i=0}^\infty \ep^i \psi_{i} \left( \frac{t}{\ep} \right)
\end{equation}
Assuming for the moment that such a representation holds,
we proceed to derive the form of the ``regular'' terms
$\phi_i(t)$ and the ``boundary layer'' terms $\psi_i(t)$.
Applying $\mathcal{L}^\ep$ to \eqref{eq:SP-expansion} 
and equating terms of $\ep$ leads to the recursive equations
\begin{align}
  \begin{aligned}
    \label{eq:phi_recurs}
    \ep^0 : & \qquad \phi_0(t) \widetilde{Q} = 0 \\
    \ep^1 : & \qquad \phi_1(t) \widetilde{Q} = 
    \frac{d \phi_0(t)}{dt} - \phi_0(t) \widehat{Q} \\
    \vdots  &  \\
    \ep^i : & \qquad \phi_i (t) \widetilde{Q} = 
    \frac{d \phi_{i-1}(t)}{dt} - \phi_{i-1}(t)\widehat{Q} 
   \end{aligned}
\end{align}
and similarly, using the ``stretched-time'' variable
$\tau=t/\ep$ one has equations for $\psi(\tau)$
\begin{align}
  \begin{aligned}
  \label{eq:psi_recurs}
  \ep^0 : & \qquad  \frac{d \psi_0(\tau)}{d \tau} =\psi_0(\tau) \widetilde{Q} \\
  \ep^1 : & \qquad \frac{d \psi_1(\tau)}{d \tau} 
  = \psi_1(\tau) \widetilde{Q} - \psi_0(\tau) \widehat{Q} \\
  \vdots & \\
  \ep^i : & \qquad \frac{d \psi_{i}(\tau)}{ d\tau} 
  =\psi_{i}(\tau) \widetilde{Q} - \psi_{i-1}(\tau) \widehat{Q}.
\end{aligned}
\end{align}
At $t=0$, 
\begin{equation}
\sum_{i=0}^\infty \ep^i 
\left( \phi_i(0) + \psi_i(0) \right) = p^\ep(0), 
\label{eq:initial_condition}
\end{equation}
so $\phi_0(0) + \psi_0(0) = p^\ep(0)$ and 
$\phi_i(0) + \psi_i(0) =0 $ for all $i \ge 1$.
Since $p^\ep(t)$ is a probability measure with 
$p^\ep(t) \cdot \one = 1$,
by sending $\ep \to 0$ in \eqref{eq:SP-expansion} 
it follows that 
\begin{align}
  \label{eq:orthog_condition}
  \phi_0(t) \cdot \one = 1
  \quad \text{ and } \quad
  \phi_i(t) \cdot \one = 0
\end{align}
for all $t \in [0,T]$ and all $i \ge 1$.

Turning to the leading regular term $\phi_0(t)$,
we note that $\phi_0(t) \cdot \widetilde{Q} =0$ is not
uniquely solvable because
$\widetilde{Q}=\diag [\widetilde{Q}^{(1)}, \dots,
\widetilde{Q}^{(\NC)} ]$
has rank $m-\NC$ 
However, writing $\phi^{(k)}_0(t)$ for the sub-vector of
$\phi_0(t)$ corresponding to fast-class $\M_k$, 
then we must have 
$\phi^{(k)}_0(t) \cdot \widetilde{Q}^{(k)} = 0$ for all $k=1, \dots, \NC$.
Since each $\widetilde{Q}^{(k)}$ is an irreducible generator, 
we then have $\phi^{(k)}_0(t) = \gamma^{(k)}(t) \widetilde{\pi}^{(k)}$
for some scalar multiplier $\gamma^{(k)}(t)$.
It can be seen (Prop 4.24\cite{yin_continuous-time_2013})
that $\gamma^{(k)}(t) = \overline{p}_k(t)$,
where $\overline{p}(t)$ is the probability measure
among fast-classes $\M_1, \dots, \M_\NC$ induced by generator 
$\overline{Q}$ (as in \eqref{eq:Q-bar-matrix}) and 
initial distribution $p^\ep(0) \cdot \widetilde{\one}
=\Prob\left\{ X^\ep(0) \in \M_k \right\} $.
This in turn determines a unique solution
for $\phi^{(k)}_0(t)$ and therefore $\phi_0(t)$.
It follows that $\phi_0(t)$ is exactly the measure $p^0_t$
in Theorem \ref{thm:prob_err_bound} induced by the 
TTS simulation procedure.

With $\phi_0(t)$ determined, 
\eqref{eq:initial_condition} 
then gives the initial condition
$\psi_0(0) = p^\ep(0) - \phi_0(0) 
= p^\ep(0) [ I_{m} - \widetilde{\one} \widetilde{\pi} ]$,
from which one can solve \eqref{eq:psi_recurs} to obtain
$\psi_0(\tau) = \psi_0(0) \cdot \exp\left\{ \widetilde{Q}\tau \right\} $.
It can be shown (Prop. 4.25\cite{yin_continuous-time_2013})
that 
\begin{equation}
  \label{eq:psi_decay}
\| \psi_0(\tau) \| \le C \exp\left\{ -\widetilde{\kappa} \tau \right\} ,
\end{equation}
where $C$ depends on the Jordan-Form of $\widetilde{Q}$.
Higher order terms can also be solved for recursively
(Prop. 4.26\cite{yin_continuous-time_2013}),
and it can be shown (Prop. 4.28\cite{yin_continuous-time_2013})
\begin{align} 
  \label{eq:SP_error}
  \sup_{0 \le t \le T}
  \| \sum_{i=0}^n \ep^i \phi_i(t) + \sum_{i=0}^n \ep^i \psi_i(t/\ep)
  - p^\ep(t) \| = O(\ep^{n+1}) .
\end{align}
In particular, using the $0$th-order expansion we have
\begin{multline}
  \label{eq:SP-0_error}
  \| p_T^0 - p^\ep_T \| = \| \phi_0(T) - p^\ep(T) \| \\
\le \| \phi_0(T) + \psi_0(T/\ep) - p^\ep(T) \| + \| \psi_0(T/\ep) \| \\
\le \| O(\ep) + O(\exp{-\widetilde{\kappa} T/\ep}) \|
\end{multline}
\qed

\subsection{Proof of Corollary \ref{cor:expectation-error} }
With the singular perturbation bound \eqref{eq:SP_error},
Corollary \ref{cor:expectation-error} follows immediately.
Since the exact process $X^\ep(t)$ is ergodic, 
there exists a time horizon $T^\ep$ such that 
$\| p^\ep(t) - \pi^\ep \| \le \ep$ for $t \ge T^\ep$.
Similarly, by \eqref{eq:psi_decay}
there exists $\widetilde{T}$ such that $\| \psi_0(t) \| \le \ep$
for $t \ge \widetilde{T}$, and 
$\| \overline{p}(t) - \overline{\pi} \| \le \ep$
for $t \ge \overline{T}$, implying that
$\|\phi_0(t) - \overline{\pi} \widetilde{\pi} \| 
= \| \overline{p}(t) \widetilde{\pi} - \overline{\pi} \widetilde{\pi} \|
\le \ep$.
Then taking $T \ge \max \{ T^\ep, \widetilde{T},
\bar{T} \}$ and applying \eqref{eq:SP_error} , we have
  \begin{multline*}
      \| p^\ep(t) - \overline{\pi}\widetilde{\pi} \|  \le 
      \|p^\ep (t) - \phi(t) - \psi(t/\ep) \| \\
      + \| \phi(t) - \overline{\pi}\widetilde{\pi} \|
      + \| \psi(t/\ep) 
      \| \le O \left( \ep +
      \exp \left\{-\widetilde{\kappa} t/\ep \right\} \right) \\
   \| \pi^\ep - \overline{\pi}\widetilde{\pi} \|
      \le \|\pi^\ep -p^\ep(T) \| 
      + \|p^\ep(T) - \overline{\pi}\widetilde{\pi} \|
      \le O(\ep)
  \end{multline*}
  and the corollary follows.
\qed

\subsection{Proof of Proposition \ref{prop:weak_convergence} }
This is a direct application of Theorem 5.27\cite{yin_continuous-time_2013},
and also follows from the error bound \eqref{eq:SP_error}.
First, one uses \eqref{eq:SP-0_error} to establish that 
\begin{multline*}
  \lim_{t \to 0} \lim_{\ep \to 0}
  \E\left[ \overline{X^\ep}(s+t) - \overline{X^\ep}(s) 
  \big| X^\ep(s)=x^{(k)}_j \right] =0
\end{multline*}
and thus $\left\{ \overline{X^\ep}(\cdot) \right\}_{\ep > 0}$ is tight.
Then one shows that the finite-dimensional distributions converge
by taking arbitrary time points 
$0 \le t_1 < \dots < t_n \le T$ and apply the Chapman-Kolmogorov equations
to see
\begin{multline*}
  \Prob\left\{ 
    \overline{X^\ep}(t_n) = \overline{y}_n 
  , \dots, \overline{X^\ep}(t_1) = \overline{y}_1 \right\} \\
  = \sum_{j_1, \dots, j_n} 
  \Prob\left\{ \overline{X^\ep}(t_n)=x^{(\overline{y}_n)}_{j_n}
  \big| \overline{X^\ep}(t_{n-1}) = x^{(\overline{y}_{n-1})}_{j_{n-1}} \right\}
  \times \dots \\
  \times 
  \Prob\left\{ \overline{X^\ep}(t_2)=x^{(\overline{y}_2)}_{j_2}
  \big| \overline{X^\ep}(t_{1}) = x^{(\overline{y}_{1})}_{j_1} \right\} \\ 
  \times \Prob\left\{ \overline{X^\ep}(t_1) = x^{(\overline{y}_1)}_{j_1} \right\}.
\end{multline*}
Then applying the error bound \eqref{eq:SP-0_error}
to each transition term to obtain
\begin{multline*}
\Prob\left\{ \overline{X^\ep}(t_l) 
= x^{(\overline{y}_l)}_{j_l} \Big| 
\overline{X^\ep}(t_{l-1}) = x^{(\overline{y}_{l-1})}_{j_{l-1}} \right\} \\
\to \Prob\left\{ \overline{X}(t_l) = \overline{y}_l
\Big| \overline{X}(t_{l-1}) = \overline{y}_{l-1} \right\}
\end{multline*}
as $\ep \to 0$, 
and thus the finite dimensional distributions converge.

\qed

\subsection{Proof of Proposition \ref{prop:sens_converge}}
Proposition \ref{prop:sens_converge} is simply the application of 
Theorem 3.2 of Ref. \citen{gupta_sensitivity_2014} to the TTS
framework. 
The method and framework for separating time-scales in
Ref. \citen{gupta_sensitivity_2014}
is slightly different than the TTS framework used here,
but it can be seen that the two are equivalent.
Here, we briefly review the multiscale framework of 
Ref. \citen{gupta_sensitivity_2014} and show how 
one can translate between their ``remainder spaces''
and the TTS ``fast-classes''.

\subsubsection{Scaling Rates and Remainder Spaces}
As in Refs.\citen{kang_separation_2013, wang_efficiency_2014},
Ref.\citen{gupta_sensitivity_2014} 
considers
reaction rates of the form
$a^N_k(x, \theta) = N^{\rho_k}\lambda_k(x, \theta)$
which scale with the ``system size''  
or ``normalization parameter'' $N \gg 0$,
where $\rho_k$ is the scaling rate for the $k$-th reaction channel.
For a given normalizing parameter $N$, the corresponding
system is denoted by $X^{N}(t)$.
One analyzes the system against a reference time scale $\gamma$
by $X^{N}_\gamma(t) = X^{N}( t N^\gamma)$.
For a given normalizing parameter $N_0 \gg 0$, the corresponding
system is denoted by $X^{N_0}(t)$.
One analyzes the system against a reference time scale $\gamma$
by $X^{N}_\gamma(t) = X^{N}( t N^\gamma)$.

The scaling rates $\rho_k$ determine the time scales at which
the reaction channels fire.
For a system with with a single level of stiffness,
there are only two scaling rates,
$\rho_{fast} > \rho_{slow}$,
 which partition
the reaction channels as either fast or slow.
Write $\Gamma_1=\{ k : \rho_k=\rho_{fast} \}$ for the fast reaction
index set, 
and similarly $\Gamma_2$ for the slow reaction index set.

Take
$\mathbb{S}_2 = 
\{ v \in \rr^d_+ : \langle v, \zeta_k \rangle = 0 
\text{ for all } k \in \Gamma_1 \}$
so that $\langle X_\gamma^N(t), v \rangle $ is unchanged by fast reactions.
Then take $\mathbb{L}_2=\text{span}(\mathbb{S}_2)$
and $\Pi_2$ be the projection map from $\rr^d$ to $\mathbb{L}_2$,
so that $\Pi_2 \zeta_k = \boldsymbol{0}$ for all $k \in \Gamma_1$.

Let $\mathbb{L}_1 = \text{span} \{ (I - \Pi_2)x : x \in \M \}$, and for 
any $v \in \Pi_2 \M$ let
$\mathbb{H}_v = \{ y \in \mathbb{L}_1 : y = (I - \Pi_2)x,
\Pi_2 x = v, x \in \M \}$, the set of remainders of elements in $\M$ 
which get projected to $v$.
Then we can define an operator $\mathbb{C}^v$ by
\begin{align*}
\mathbb{C}^v f (z) = \sum_{k \in \Gamma_1} 
\lambda_k(v +z, \theta) [ f(z + \zeta_k) - f(z) ] 
\end{align*} 
which is a generator of a Markov chain with state space
$\mathbb{H}_v$ (note that $y \in \mathbb{H}_v \implies
y + \zeta_k \in \mathbb{H}_v$ for all $k \in \Gamma_1$).

Assuming $\mathbb{H}_v$ under $\mathbb{C}^v$ is ergodic,
there is a stationary distribution $\pi^v$.
Then for each slow reaction $k \in \Gamma_2$ 
one can define the ``averaged'' propensities
$\hat{\lambda}_k(v, \theta) = \int_{\mathbb{H}_v} 
\lambda_k(v + z, \theta) \pi^v(dz)$ for all $v \in \Pi_2 \M$ .
Using the random time change representation,
define the Markov chain on $\Pi_2 \M$ by
\begin{align}
\label{eq:Xhat_defn}
\hat{X}_\theta(t) = \Pi_2 x_0 + \sum_{k \in \Gamma_2} 
Y_k \left( \int_0^t \hat{\lambda}_k( \hat{X}(s), \theta) ds \right) 
\Pi_2 \zeta_k
\end{align}
Taking $\gamma_2=-\rho_{slow}$ as the slow time scale, 
one has
$\Pi_2 X^N_{\gamma_2, \theta} \Rightarrow \hat{X}_\theta$ as 
$N \to \infty$\cite{kang_separation_2013} under more general conditions
than Assumptions \ref{assum:finite_states},\ref{assum:recurr_states}.
Under this context, 
Theorem 3.2 of Ref.\citen{gupta_sensitivity_2014}
states that
\begin{align}
\label{eq:gupta_scaling_sens}
\lim_{N \to \infty} \frac{\partial}{\partial \theta} 
\E \left\{ f(X^N_{\gamma_2, \theta}(t)) \right\}
= \frac{\partial}{\partial \theta}
\E \left\{ f_\theta (\hat{X}_\theta(t) ) \right\}
\end{align}
where
$f_\theta(v) = \int_{\mathbb{H}_v} f(v+y) \pi^v_\theta(dy)$.

\subsubsection{Equivalence of Fast-Classes and Remainder Spaces}
Here we show how the TTS framework is equivalent to the scaling rate 
framework.
Consider a TTS reaction network as described by \eqref{eq:Q-ep_decomp}.
Taking $N=1/\ep$, $\rho_{fast}=1$, $\rho_{slow}=0$, 
it is easy to see that 
\begin{align*}
  X^N_{0,\theta }(t) = X^N_\theta (t) 
   \overset{\mathcal{D}}{=} X^\ep(t) .
\end{align*}
so $\lim_{N\to\infty}X^N_{0, \theta}(t) 
= \lim_{\ep\to0}X^\ep(t)$.
It remains to identify $f_\theta \left( \hat{X}_\theta(t) \right) $
from \eqref{eq:Xhat_defn}, \eqref{eq:gupta_scaling_sens}
 with $\overline{f}\left(\overline{X}(t)\right)$ from \eqref{eq:fbar_defn}.
 We do so by showing the equivalence of the fast-classes 
 $\M_l$ and the remainder spaces $\mathbb{H}_v$.
 
  \begin{lemma}
   The projection map $\Pi_2$ is invariant on fast-classes $\M_l$. 
   The set of remainder spaces 
   $\left\{ \mathbb{H}_v : v \in \Pi_2 \M \right\} $
   is in one-to-one correspondence with the set of fast-classes
   $\left\{ \M_l \right\} $.
   Additionally, each $x \in \M_l$
   corresponds to a unique element
   $y \in \mathbb{H}_v$ for some $v \in \Pi_2\M$.
 \end{lemma}

 \begin{proof}
   Define $\eta: \left\{ \M_l \right\} \to 
   \left\{ \mathbb{H}_v : v \in \Pi_2 \M \right\} $ by
   \begin{align*}
     \eta(\M_l) = \mathbb{H}_{\Pi_2(x)} 
     \qquad \text{for any  } x \in \M_l .
   \end{align*}
   Then $\eta$ is well-defined, since $x, y \in \M_l$ implies
   that $y = x + \sum_{k \in \Gamma_1} c_k \zeta_k$ for some 
   $c_k \in \mathbb{N}$, and $\Pi_2(\zeta_k) = \boldsymbol{0}$ for all
   $k \in \Gamma_1$ gives $\Pi_2(y) = \Pi_2(x)$.
   Clearly, $\eta$ is also onto. 
   
   It remains to establish $\eta$ is injective.
   It is sufficient to show that 
   if $\Pi_2(x) = v$ and $\Pi_2(x')=v$ for $x, x' \in \M$,
   then $x$ and $x'$ belong to the same fast-class $\M_l$.
   Since $\Pi_2$ projects onto the span of the complement of
   $\operatorname{span} \left\{ \zeta_k: k \in \Gamma_1 \right\}$,
   we have 
   $v-x =\sum_{k \in \Gamma_1} c_k \zeta_k$
   and $v-x' = \sum_{k \in \Gamma_1} c'_k \zeta_k$
   for some $c_k, c'_k \in \mathbb{R}$. 
   Then $x' = x + \sum_{k \in \Gamma_1} (c_k - c'_k) \zeta_k$,
   with $x', x, \zeta_k \in \mathbb{N}^d$, so it follows that
   $(c_k -c'_k) \in \mathbb{N} $ for all $k$.
   Hence $x$ and $x'$ communicate by fast reactions and thus belong
   to the same fast-class $\M_l$. Therefore, 
   $\Pi_2$ is invariant on fast-classes and $\eta$ is injective.
   
   Finally, since $\Pi_2(\cdot)$ is invariant on fast-classes
   $\M_l$, it follows that $x \to x-\Pi_2(x)$ bijectively
   maps elements of $\cup_l \M_l$ to elements of 
   $\cup_{v \in \Pi_2(\M)} \mathbb{H}_v$ such that
   $x, x' \in \M_l$ implies $(x - \Pi_2(x)), (x'-\Pi_2(x')) 
   \in \mathbb{H}_{\Pi_2(\M_l)}$.
 \end{proof}

Because of the direct correspondence between
$\left\{ \Pi_2(x) : x \in \M \right\} $
and 
$\left\{ \M_l \right\}$,
we see (upon reordering states) that
for $v=\Pi_2(\M_l)$, we have
$\pi^v_\theta = \widetilde{\pi}^{(l)}_\theta$,
so that $\hat\lambda_k(v, \theta) 
= \overline{\lambda}_k (\M_l, \theta)$,
and $f_\theta(v) = \overline{f}(\M_l, \theta)$.
Thus, $f_\theta \left( \hat{X}(t) \right) $
has the same distribution as 
$\overline{f}\left( \M_l, \theta \right)$
and so 
$\frac{\partial}{\partial \theta}
\E\left\{ f_\theta\left( \hat{X}_\theta(t) \right) \right\}
= \frac{\partial}{\partial \theta} 
\E\left\{ \overline{f} 
\left( \overline{X}_\theta (t) \right) \right\}$.
Therefore, using \eqref{eq:gupta_scaling_sens} we have
\begin{align}
\begin{aligned}
\lim_{\ep \to 0} \partial_\theta 
\E \left\{ f\left( X^\ep(t) \right) \right\}
=
\lim_{N \to \infty} \partial_\theta 
\E \left\{ f\left(X^N_{0,\theta}(t) \right) \right\} \\
=
\partial_\theta \E \left\{
f_\theta \left( \hat{X}_\theta (t) \right) \right\}
=
\partial_\theta \E \left\{
\overline{f} \left( \overline{X}(t) \right) \right\}
\end{aligned}
\end{align}
and hence Proposition \ref{prop:sens_converge}.

\section{Analytic solution of the Model System}
\label{sec:AnalylicSolnExample}

{In a well-mixed system with linear
propensities, the time-evolution of the system can be obtained from a set of ordinary 
differential equations (ODE). 
The single time-scale
(STS) system can be modeled with 
a system of
ODEs. The two time-scale (TTS) system imposes
algebraic constraints for the fast modes, resulting in
an algebraic 
differential system of equations. In both cases, a set of adjoint ODEs can be used to compute sensitivities alongside species populations.}

In our model system, gaseous species A adsorbs onto a catalyst surface, isomerizes to species B, and then desorbs. A diagram of the reaction network is shown in Figure \ref{fig:networkpic}. The reactions along with their rate laws are shown in Table \ref{tab:toyrxns}. $N_A$, $N_B$, and $N_*$ denote the surface coverages of species A, B, and empty sites respectively. The adsorption/desorption of species A is assumed to be much faster than the others. The separation of time scales is captured with the dimensionless parameter $\ep << 1$. The system contains $M=3$ species and $R=6$ reactions. Mathematically, we use the $M\times 1$ column vector $N$ to specify the species populations, where $N_1=N_A$, $N_2=N_B$, and $N_3=N_*$.    

\begin{figure}
\begin{minipage}[b]{.5\linewidth}
	\centering
	\includegraphics[width=0.8\textwidth]{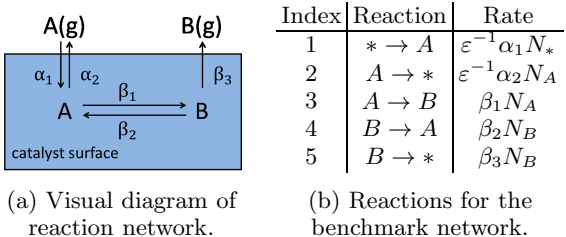}
	\subcaption{Visual diagram of reaction network.}
	\label{fig:networkpic}
\end{minipage}%
\begin{minipage}[b]{.5\linewidth}
	\centering
	\begin{tabular}{c | c | c }
	
	    Index & Reaction & Rate \\
	    	\hline
		1 & $*\rightarrow A$ & $\ep^{-1} \alpha_1 N_*$ \\
		2 & $A\rightarrow *$ & $\ep^{-1} \alpha_2 N_A$ \\
		3 & $A\rightarrow B$ & $\beta_1 N_A $ \\
		4 & $B\rightarrow A$ & $\beta_2 N_B$ \\
		5 & $B\rightarrow *$ & $\beta_3 N_B$ \\
	  \end{tabular}
	\subcaption{Reactions for the benchmark network.}
	\label{tab:toyrxns}
\end{minipage}
\caption{Description of model chemical reaction network.}
\label{fig:networkdescription}
\end{figure}

{The linear dependence of the reaction rates is written as}
\begin{equation}
r(N) = \left[ \begin{array}{c}
\ep^{-1} \alpha_1 N_3 \\
\ep^{-1} \alpha_2 N_1 \\
\beta_1 N_1 \\
\beta_2 N_2 \\
\beta_3 N_2  \end{array} \right].
\end{equation}

Each row of the stoichiometric matrix corresponds to a different species, which are $N_1$, $N_2$, and $N_3$ respectively. The columns correspond to each of reactions 1-5 in order. Extracting the information from Table \ref{tab:toyrxns} and putting it in mathematical form gives the $M\times R$ 
{stoichiometric} matrix
\begin{equation}
S=\left[ \begin{array}{ccccc}
1 & -1 & -1 & 1 & 0 \\
0 & 0 & 1 & -1 & 1 \\ 
-1 & 1 & 0 & 0 & -1 \end{array} \right]
\end{equation}

{The transformation matrix} 
\begin{equation}
T = \left[ \begin{array}{ccc}
1 & 0 & 0 \\
0 & 1 & 0 \\ 
1 & 1 & 1 \end{array} \right].
\end{equation}
{yields $y=T\cdot N$ and} 
$T$ can be decomposed into $T_f = \left[ \begin{array}{ccc}
1 & 0 & 0 \end{array} \right]$ and $T_s = \left[ \begin{array}{ccc}
1 & 1 & 0 \\ 
0 & 0 & 1 \end{array} \right]$ {for the slow modes} by looking at the 0 rows of $S'_f$. This gives us the transformed variables as $y_f=\left[ \begin{array}{c}
N_1\end{array} \right]$ and $y_s=\left[ \begin{array}{c}
N_2 \\ 
N_1+N_2+N_3 \end{array} \right]$.

In the context of our example problem, we can assign physical meaning to the transformation:
The variable $y_1=N_A$ is affected by both slow and fast reactions. For a given set of slow variables, we can solve for $y_1$ to specify the equilibrium constraint of $r_1=r_2$. The variable $y_2=N_B$ is unaffected by the fast adsorption/desorption of A, but is affected by the slow reactions.
Finally, the variable $y_3=N_A+N_B+N_*$ is a second "slow" variable. In this example $y_3=1$ applies at all times due to stoichiometric constraints. However, it is still identified as a "slow mode" because this constraint is not a consequence of disparities in reaction time scales.

The system is simulated with the choice of parameters $N_0=\left[ \begin{array}{c}
30 \\
60 \\ 
10 \end{array} \right]$, $\alpha_1=1$, $\alpha_2=1.5$, $\beta_1=2$, $\beta_2=1$, $\beta_3=0.4$, $\ep=0.01$. The simulation results are shown in Figure \ref{fig:A}. {Table \ref{tab:table2} shows values at $t=1.3$s and $t=100$s along with CLR and CELR estimates with statistical confidence intervals.} Derivatives with respect to $\beta_2$ and $\beta_3$ overlap because both affect system properties through the the independent parameter $\beta_2+\beta_3$. 

 In general, the system parameters need not be the rate constants themselves. A different parameterization would involve a transformation of the rate constants. Sensitivities could be obtained through a chain rule. 

\begin{figure}

	\centering
	\includegraphics[width=1.0\linewidth]{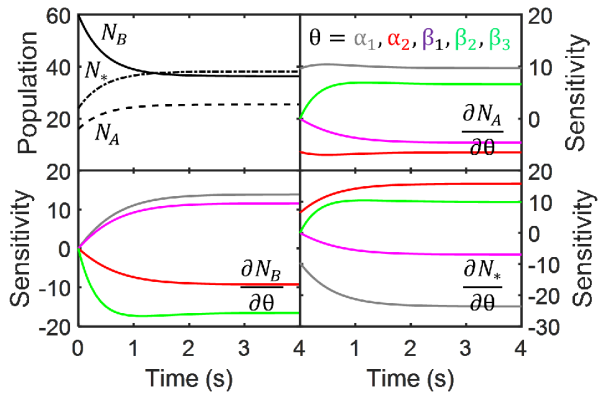}
\caption{Network results for the two time-scale system. Graphs show population counts (top left), derivatives of species A (top right), derivatives of species B (bottom left), and derivatives of empty sites (bottom right).}
\label{fig:A}
\end{figure}

\begin{table}[h!]
  \begin{center}
    \caption{Comparison of the CLR and CELR estimators with the ODE solution. Values show actual values rather than errors. Values in the table refer to the sensitivity of the species B with respect to the parameter given by the row label. 95\% confidence intervals are based on statistical noise.}
    \label{tab:table2}
    \begin{tabular}{cccc}
    \hline
    & \multicolumn{3}{c}{$t=1.3$s}\\
    & ODE & CLR & CELR\\
    \hline 
    $\alpha_1$ & 11.9 & $11.8 \pm 1.5$ & $7.2 \pm 0.9$\\ 
    $\alpha_2$ & -7.9 & $-7.7 \pm 1.0$ & $-4.7 \pm 0.6$\\
    $\beta_1$ & 9.9 & $10.0 \pm 1.2$ & $6.0 \pm 0.8$\\
    $\beta_2$ & -17.4 & $-17.2 \pm 2.4$ & $-14.5 \pm 1.7$\\ 
    $\beta_3$ & -17.4 & $-17.1 \pm 3.4$ & $-13.3 \pm 2.4$\\ 
    \hline 
    & \multicolumn{3}{c}{$t=100$s}\\
    & ODE & CLR & CELR\\
    \hline
    $\alpha_1$ & 13.9 & $2.3 \pm 13.7$ & $13.9 \pm 1.4$\\
    $\alpha_2$ & -9.3 & $-3.4 \pm 8.9$ & $-9.2 \pm 1.0$\\
    $\beta_1$ & 11.6 & $4.1 \pm 10.6$ & $11.7 \pm 1.2$\\
    $\beta_2$ & -16.5 & $-12.8 \pm 18.1$ & $-16.1 \pm 2.0$\\
    $\beta_3$ & -16.5 & $-19.9 \pm 29.0$ & $-16.5 \pm 3.2$
    
    \end{tabular}
  \end{center}
\end{table}

%

%

\end{document}